\documentclass[12pt,centertags,oneside]{amsart}
\usepackage{amsmath,amstext,amsthm,amscd,typearea}
\usepackage{amssymb}
\usepackage{fancyhdr}
\usepackage[mathscr]{eucal}
\usepackage{mathrsfs}
\usepackage{concmath}
\usepackage{charter} 
\usepackage{dsfont}          
\usepackage{hyperref}       
\usepackage{color}      
\usepackage{titletoc}

\usepackage[a4paper,width=16.2cm,top=2.8cm,bottom=2.8cm]{geometry} 
  
\numberwithin{equation}{section}     
\newcommand{\field}[1]{\mathbb{#1}} 
\newcommand{\Z}{\field{Z}}
\newcommand{\R}{\field{R}}   
\newcommand{\C}{\field{C}} 
\newcommand{\N}{\field{N}}

\def\cC{\mathscr{C}}

\def\cL{\mathscr{L}}

\def\mL{\mathcal{L}}

\def\mT{\mathcal{T}}

\newcommand{\til}[1]{\widetilde{#1}}

\newcommand{\cali}[1]{\mathscr{#1}}

\newcommand{\cE}{\cali{E}} 

\DeclareMathOperator{\Ker}{Ker}

\DeclareMathOperator{\Dom}{Dom}

\newcommand{\abs}[1]{\lvert#1\rvert}

\newcommand{\ol}{\overline}
\newcommand{\ddbar}{\overline\partial}
\newcommand{\dbar}{\partial}
\newtheorem{thm}{Theorem}[section]
\newtheorem{lemma}[thm]{Lemma}
\newtheorem{prop}[thm]{Proposition}
\newtheorem{cor}[thm]{Corollary}
\theoremstyle{definition}
\newtheorem{rem}[thm]{Remark}
\theoremstyle{definition}
\newtheorem{defn}[thm]{Definition}

\newtheorem{exam}[thm]{Example}

\newtheorem{ass}[thm]{Assumption}

\newcommand{\be}{\begin{eqnarray}}
\newcommand{\ee}{\end{eqnarray}}
\newcommand{\ov}{\overline}
\newcommand{\ovz}{\overline{z}}

\newcommand{\oh}{\widehat}

\newcommand{\comment}[1]{}
\begin{document}        
\title 
{Szeg\H{o} kernel asymptotics on some non-compact complete CR manifolds}   
\author{Chin-Yu Hsiao, George Marinescu,
Huan Wang}  
 
\address{Chin-Yu Hsiao, Institute of Mathematics, Academia Sinica, Taipei, Taiwan}
\email{chsiao@math.sinica.edu.tw} 

\address{George Marinescu, Universit{\"a}t zu K{\"o}ln,  Mathematisches Institut, Weyertal 86-90,   50931 K{\"o}ln, Germany}
\email {gmarines@math.uni-koeln.de}  

\address{Huan Wang, Institute of Mathematics, Academia Sinica, Taipei, Taiwan}     
\email {huanwang2016@hotmail.com}  
  
\date{21. Dec. 2020}     
               
\begin{abstract}    
We establish Szeg\H{o} kernel asymptotic expansions on 
non-compact strictly pseudoconvex complete CR manifolds 
with transversal CR $\R$-action under  certain natural geometric conditions.
\end{abstract}       
\maketitle     
\tableofcontents           
  
\section{Introduction}   
Let $(X, T^{1,0}X)$ be a CR manifold of dimension $2n+1$, $n\geq1$.  
The orthogonal projection $S^{(q)}:L^2_{0,q}(X)\rightarrow 
{\rm Ker\,}\Box^{(q)}_b$ onto ${\rm Ker\,}\Box^{(q)}_b$
is called the Szeg\H{o} projection, while its distribution kernel 
$S^{(q)}(x,y)$ is called the Szeg\H{o} kernel, where 
$\Box^{(q)}_b$ denote the Kohn Laplacian acting on $(0,q)$ forms.
The study of the Szeg\H{o} kernel is a classical subject in several complex 
variables and CR geometry.
When $X$ is strictly pseudoconvex, compact and $\Box^{(0)}_b$ has closed range,  
Boutet de Monvel-Sj\"ostrand~\cite{BS76} showed that $S^{(0)}(x,y)$
is a complex Fourier integral operator. The Boutet de Monvel-Sj\"ostrand 
description of the Szeg\H{o} kernel had a profound impact
in several complex variables, 
symplectic and contact geometry, geometric quantization and  K\"ahler geometry. 
These ideas also partly motivated the introduction of the recent direct 
approaches and their various extensions, 
see \cite{MM, MM08}.

However, almost all the results on Szeg\H{o} kernel assumed 
that $X$ is compact, while for non-compact complex manifolds
the Bergman kernel asymptotics was  
comprehensively studied \cite{MM,MM08,MM15}. 
Therefore, it is interesting to look for counterparts for the Szeg\H{o} kernel 
on non-compact CR manifolds. 
Let us see some simple examples and describe our motivation briefly. 
On $\mathbb C^n$, consider the hypersurface
\[Y:=\{z=(z_1,\ldots,z_n)\in\mathbb C^n;\, {\rm Im\,}z_n=f(z_1,\ldots,z_{n-1})\},\]
where $f\in\cC^\infty(\mathbb C^{n-1},\mathbb R)$. 
 Then, $Y$ is a non-compact CR manifold. There are many smooth CR functions on $Y$. 
 Even in this simple example, we do not know the behavior 
 of the associated Szeg\H{o} kernel. Let us see another example. 
 Consider $\field{H}=\C^n\times \R$ with  CR structure 
\be
T^{1,0}\field{H}:={\rm span\,}\left\{\frac{\dbar}{\dbar z_j}+
i\frac{\dbar\phi}{\dbar z_j}(z)\frac{\dbar}{\dbar x_{2n+1}}\right\}_{j=1}^n,
\ee
where $\phi\in\cC^\infty(\C^n,\R)$. Then, $\field{H}$ is also 
a non-compact CR manifold and the Szeg\H{o} kernel has been studied when 
$\phi$ is quadratic (see~\cite{HHL18}) but for general $\phi$, there are fewer results. 
Both $Y$ and $\field{H}$ are non-compact CR manifolds with transversal 
CR $\mathbb R$-action. Therefore, we think that the study of the Szeg\H{o} kernels 
on non-compact CR manifolds with transversal CR $\mathbb R$-action is a 
very natural and interesting question. 
	
In~\cite{Hsi:10}, the first author obtained the Szeg\H{o} kernel 
asymptotic expansion on the non-degenerate part of the Levi form 
under the assumption that Kohn Laplacian has closed range. 
It should be mentioned that the method in~\cite{Hsi:10} works well 
for non-compact setting. But for general non-compact CR manifolds, 
closed range property is not a natural assumption since in the Heisenberg case 
mentioned above,  even for $\phi$ quadratic, $\Box^{(0)}_b$ 
does not have closed range but we still have Szeg\H{o} kernel asymptotic expansion. 
In this paper, we show that  $\Box^{(0)}_b$ has local closed range with respect to 
some operator $Q_\lambda$ (see Definition~\ref{d-gue201114yydf}) 
under certain geometric conditions. Furthermore, combining
this local closed range property and by more detailed analysis, 
we establish Szeg\H{o} kernel asymptotic expansions on non-compact 
strictly pseudoconvex complete CR manifolds with transversal CR $\R$-action 
under  certain natural geometric conditions.
To have local closed range property, we established the CR Bochner-Kodaira-Nakano 
formula analogue to \cite{Dem:85}, see Theorem \ref{t-gue201028yyd}, 
which has its own interest. {This is also a refinement of Tanaka's basic identities \cite[Theorems 5.1, 5.2]{Tan75} in our context.} We remark that the results in this paper hold 
both for transversal CR $\R$-action and $S^1$-action.

 
We  now formulate our main result.  
We will work in the following setting.
\begin{ass}\label{as1}
$(X,HX,J,\omega_0)$ is an orientable strictly pseudoconvex
CR manifold of dimension $2n+1$, $n\geq 1$, 
where $HX$ is the Levi distribution, $J$ is the complex structure
and $\omega_0$ is a contact form, endowed with
a smooth transversal CR $\R$-action on $X$ 
preserving $\omega_0$ and $J$.  
\end{ass} 
We denote by $T^{1,0}X$ and $T^{0,1}X$ the bundles of tangent
vectors of type $(1,0)$ and $(0,1)$, respectively.
Since there is a transversal and 
CR $\mathbb R$-action (see Definition~\ref{d-gue201128yyd}).
we have a decomposition $\C TX=T^{1,0}X\oplus T^{0,1}X\oplus\C T$,
where $T$ is the infinitesimal generator of the $\R$-action.
We can and will choose the $\R$-invariant contact form $\omega_0$ such that
\begin{equation}\label{e-gue201128yyd1}
{\omega_0(T)=1},
\end{equation}
that is, $T$ is also the characteristic vector field of the CR manifold $X$.
{Since $X$ is strongly pseudoconvex, we can take $T$ and $\omega_0$ so that $\frac{1}{2i}d\omega_0|_{T^{1,0}X}$ is positive definite.}
There exists an $\R$-invariant Hermitian metric $g=g_X$ on $\C TX$ 
such that 
\begin{equation}\label{e-gue201128yyd}
T^{1,0}X\perp T^{0,1}X,\:\:T\perp(T^{1,0}X\oplus T^{0,1}X),\:\:
\langle\,T\,|\,T\,\rangle_g=1. 
 \end{equation}
Given such a metric we will denote by $\Theta_X$
its fundamental $(1,1)$-form given by
$\Theta_X(a,\ov b)=\sqrt{-1}\langle a, b\rangle_g$ for $a,b\in T^{1,0}X$.

 
Denote by $\cL$ the Levi form with respect to $\omega_0$ (see \eqref{e-201120levi}). 
Since $X$ is strictly pseudoconvex, $\cL$ induces a Hermitian metric, called 
the Levi (or Webster) metric, 
\begin{equation}\label{eq:LeviMetric}
{g_{\cL}=d\omega_0(\cdot,J\cdot)+\omega_0(\cdot)\omega_0(\cdot),}
\end{equation}  
on $TX$ and by extension on $\C TX$
with the properties \eqref{e-gue201128yyd}.
We shall also denote the
inner product given by this metric by
$\langle\,\cdot\,|\,\cdot\,\rangle_{\cL}$ on $\mathbb CTX$ 
so that $\langle\,u\,|\,v\,\rangle_{\cL}=2\cL(u,\ol v)$, $u, v\in T^{1,0}X$ 
and \eqref{e-gue201128yyd} hold. 
Let $K^*_X:=\det(T^{1,0}X)$ and let $R^{K^*_X}_{\cL}$ be the curvature of 
$K^*_X$ induced by $\langle\,\cdot\,|\,\cdot\,\rangle_{\cL}$ 
(see \eqref{e-gue201025yyd} and \eqref{e-201120curc}). 
 
 Let $(\,\cdot\,|\,\cdot\,)$ be the $L^2$ inner product on $\Omega^{0,q}_c(X)$ 
 induced by $g_X$ and let $L^2_{0,q}(X)$ be the completion 
 of $\Omega^{0,q}_c(X)$ with respect to 
 $(\,\cdot\,|\,\cdot\,)$. We write $L^2(X):=L^2_{0,0}(X)$. Let 
 \[S^{(0)}: L^2(X)\rightarrow{\rm Ker\,}\ddbar_b\subset L^2(X)\]
 be the orthogonal projection with respect to $(\,\cdot\,|\,\cdot\,)$ 
 and let $S^{(0)}(x,y)\in\mathscr D'(X\times X)$ be the distribution kernel of 
 $S^{(0)}$, where $\ddbar_b$ denotes the tangential Cauchy-Riemann operator 
 (see Definition~\ref{d-gue201204yyd}). Let $D\subset X$ be an open set and let 
 $A, B: \cC^\infty_c(D)\rightarrow\mathscr D'(X)$ be continuous operators. 
 We write $A\equiv B$ on $D$ if $A-B$ is a smoothing operator on $D$. 
The main result of this article is as follows.
    
\begin{thm}\label{thm-main}
Let $(X,T^{1,0}X)$ be a strictly pseudoconvex CR manifold of dimension $2n+1$, 
$n\geq 1$, with a transversal CR $\R$-action on $X$. 
Assume that the Levi metric $g_\cL$ is complete
and there is $C>0$ such that 
\be	\label{e-cond}
\sqrt{-1}R^{K^*_X}_{\cL}\geq-2C\sqrt{-1}\cL\,,\quad 
{(2\sqrt{-1}\cL)^n\wedge\omega_0\geq C\Theta_X^n\wedge\omega_0.}
\ee
	Let $D\Subset X$ be a local coordinate patch with local coordinates $x=(x_1,\ldots,x_{2n+1})$. 
	Then, 
	\begin{equation}\label{e-gue201204yyda}
		S^{(0)}(x,y)\equiv\int^\infty_0e^{i\varphi(x,y)t}s(x,y,t)dt\ \ \mbox{on $D$},
	\end{equation}
	where $\varphi\in\cC^\infty(D\times D)$, 
	\begin{equation}\label{e-gue201204yydb}
		\begin{split}
			&\varphi\in \cC^\infty(D\times D),\ \ {\rm Im\,}\varphi(x, y)\geq0,\\
			&\varphi(x, x)=0,\ \ \varphi(x, y)\neq0\ \ \mbox{if}\ \ x\neq y,\\
			&d_x\varphi(x, y)\big|_{x=y}=\omega_0(x), \ \ d_y\varphi(x, y)\big|_{x=y}=-\omega_0(x), \\
			&\varphi(x, y)=-\ol\varphi(y, x),
		\end{split}
	\end{equation}
	$s(x, y, t)\in S^{n}_{{\rm cl\,}}\big(D\times D\times\mathbb{R}_+\big)$
	and the leading term $s_0(x,y)$ of the expansion \eqref{e-gue201114yydII} 
	of $s(x,y,t)$ satisfies
	\begin{equation}
		s_0(x, x)=\frac{1}{2}\pi^{-n-1}
		\abs{{\rm det\,}{\cL}_x},\ \ \mbox{for all $x\in D$}.
	\end{equation}
\end{thm}  

We refer the reader to Definition~\ref{d-gue201114yyd} for the meaning of the space $S^{n}_{{\rm cl\,}}\big(D\times D\times\mathbb{R}_+\big)$. We also refer the reader to~\cite[Theorems 3.3, 4.4]{HM16} for more properties for the phase $\varphi$ in \eqref{e-gue201204yyda}.

We now apply our main result to complex manifolds. 
Let us recall some notations and terminology. Let $(L,h^L)$ 
be a holomorphic line bundle over a Hermitian manifold $(M,\Theta_M)$, 
where $h^L$ denotes a Hermitian metric on $L$ and $\Theta_M$ 
is a positive $(1,1)$ form on $M$. 
For every $k\in\mathbb N$, let $(L^k,h^{L^k})$ be the $k$-th power of $(L,h^L)$.
The positive $(1,1)$ form $\Theta_M$ and $h^{L^k}$  induces 
a $L^2$ inner product $(\,\cdot\,|\,\cdot\,)_{\Theta_M}$ on 
$\Omega^{0,q}_c(M,L^k))$. Let $L^2_{0,q}(M,L^k)$ be the completion 
of $\Omega^{0,q}_c(M,L^k)$ with respect to $(\,\cdot\,|\,\cdot\,)_{\Theta_M}$. We write $L^2(M,L^k):=L^2_{0,0}(M,L^k)$. Let 
\[H^0_{(2)}(M,L^k)={\rm Ker\,}\ddbar_k:=\{u\in L^2(M,L^k);\, \ddbar u=0\},\]
be the space of holomorphic square integrable sections of $L^k$. 
Let $\{f_j\}_{j=1}^{{d_k}}$ be an orthonormal basis for $H^0_{(2)}(M,L^k)$ 
with respect to $(\,\cdot\,|\,\cdot\,)_{\Theta_M}$, 
where $d_k\in\mathbb N\bigcup\{\infty\}$. 
Let $s$ be a local holomorphic trivializing section of $L$ 
defined on an open set $D\Subset M$, $|s|^2_{h^L}=e^{-2\phi}$, 
$\phi\in\cC^\infty(D,\mathbb R)$. 
On $D$, we write $f_j:=s^k\otimes\tilde f_j$, $\tilde f_j\in\cC^\infty(D)$, 
$j=1,\ldots,d_k$. The localized Bergman kernel on $D$ is given by 
\begin{equation}\label{e-gue201205ycd}
P_{k,s}(x,y):=\sum^{d_k}_{j=1}e^{-k\phi(x)}
\tilde f_j(x)\overline{\tilde f_j(y)}e^{-k\phi(y)}\in\cC^\infty(D\times D). 
\end{equation} 

Let $R^L$ be the Chern curvature of $L$ induced by $h^L$. 
Assume that $\omega=\sqrt{-1}R^L$ is positive. 
Let $K^*_M:=\det(T^{1,0}M)$ and let $R^{K^*_M}_{\omega}$ 
be the curvature of $K^*_M$ induced by $\omega$. 
Applying Theorem~\ref{thm-main} to the circle bundle of $(L,h^L)$, we get: 

\begin{thm}\label{t-gue201205yyd}
Let $(L,h^L)$ be a Hermitian holomorphic line bundle over a Hermitian manifold 
$(M,\Theta_M)$. 
We assume that $\omega=\sqrt{-1}R^L$ is positive and defines a 
complete K\"ahler metric on $M$.
We assume moreover that and there is $C>0$ 
such that 
\begin{equation}\label{e-gue201205yydr}
\sqrt{-1}R^{K^*_M}_{\omega}\geq-C\omega,\quad 
\omega^n\geq C\Theta_M^n \quad\text{on $M$}.
\end{equation}
Let $s$ be a local holomorphic trivializing section of $L$ defined on 
an open set $D\Subset M$. Then, 
\begin{equation}\label{e-gue201205yyds}
P_{k,s}(x,y)\equiv e^{ik\Phi(x,y)}b(x,y,k)\mod O(k^{-\infty})\ \ \mbox{on $D$}, 
\end{equation}
where $\Phi\in \cC^\infty(D\times D)$, ${\rm Im\,}\Phi(x, y)\geq C|x-y|^2$, 
$C>0$, $\Phi(x,x)=0$, for every $x\in D$, 
\begin{equation}\label{e-gue201205yydw}
b(x,y,k)\in S^n_{{\rm loc\,}}(1;D\times D),\ \ 
b(x,y,k)\sim\sum^{\infty}_{j=0}k^{n-j}b_j(x,y)\ \ 
\mbox{in $S^n_{{\rm loc\,}}(1;D\times D)$},
\end{equation}
$b_j(x,y)\in\cC^\infty(D\times D)$, $j=0,1,\ldots$\,, 
\[b_0(x,x)=(2\pi)^{-n}\frac{\omega^n(x)}{\Theta^n_M(x)},\ \ 
\mbox{for every $x\in D$}.\]
\end{thm}

We refer the reader to~\cite[Section 3.3]{HM14} for the precise 
meanings of $A\equiv B\mod O(k^{-\infty})$ on $D$, 
$S^n_{{\rm loc\,}}(1;D\times D)$ and 
the asymptotic sum in \eqref{e-gue201205yydw}. 

\begin{rem}
Our conditions \eqref{e-gue201205yydr} are different from~\cite[Theorem 6.1.1]{MM}. In particular, we do not need that the uniform bounded condition for $\dbar\Theta_M$. The reason is that we can get asymptotic expansion under local closed range condition instead of standard closed range or spectral gap condition. 
\end{rem}

\section{Preliminaries}\label{sec_pre} 
We use the following notations through this article: $\mathbb N=\{1,2,\ldots\}$ is the set of natural numbers, $\mathbb N_0=\mathbb N\bigcup\{0\}$, $\mathbb R$ is the set of 
real numbers, $\overline{\mathbb R}_+=\{x\in\mathbb R;\, x\geq0\}$. 
For $m\in \N$, let $x=(x_1,\ldots,x_m)$ be coordinates of $\R^m$. For $n\in \N$, 
let $z=(z_1,\ldots,z_n)$, $z_j=x_{2j-1}+\sqrt{-1}x_{2j}$, $j=1,\ldots,n$, be coordinates of $\C^n$. We write
\be
\frac{\dbar}{\dbar z_j}:=\frac{1}{2}\left(\frac{\dbar}{\dbar x_{2j-1}}-\sqrt{-1}\frac{\dbar}{\dbar x_{2j}}\right), \quad \frac{\dbar}{\dbar \ov z_j}:=\frac{1}{2}\left(\frac{\dbar}{\dbar x_{2j-1}}+\sqrt{-1}\frac{\dbar}{\dbar x_{2j}}\right),
\ee
\be
dz_j=dx_{2j-1}+\sqrt{-1}dx_{2j}, \quad d\ov z_j=dx_{2j-1}-\sqrt{-1}dx_{2j}. 
\ee  
 
Let $X$ be a $\cC^\infty$ paracompact manifold.
We let $TX$ and $T^*X$ denote the tangent bundle of $X$
and the cotangent bundle of $X$ respectively.
The complexified tangent bundle of $X$ and the complexified cotangent bundle of $X$ are be denoted by $\mathbb CTX$
and $\mathbb C T^*X$, respectively. Write $\langle\,\cdot\,,\cdot\,\rangle$ to denote the pointwise
duality between $TX$ and $T^*X$.
We extend $\langle\,\cdot\,,\cdot\,\rangle$ bilinearly to $\mathbb CTX\times\mathbb C T^*X$.

Let $D\subset X$ be an open set . The spaces of distributions of $D$ and
smooth functions of $D$ will be denoted by $\mathscr D'(D)$ and $\cC^\infty(D)$ respectively.
Let $\mathscr E'(D)$ be the subspace of $\mathscr D'(D)$ whose elements have compact support in $D$. 
Let $\cC^\infty_c(D)$ be the subspace of $\cC^\infty(D)$ whose elements have compact support in $D$.
Let $A: \cC^\infty_c(D)\rightarrow \mathscr D'(D)$ be a continuous map. We write $A(x, y)$ to denote the distribution kernel of $A$.
In this work, we will identify $A$ with $A(x,y)$. 
The following two statements are equivalent
\begin{enumerate}
\item $A$ is continuous: $\mathscr E'(D)\rightarrow \cC^\infty(D)$,
\item $A(x,y)\in\cC^\infty(D\times D)$.
\end{enumerate}
If $A$ satisfies (a) or (b), we say that $A$ is smoothing on $D$. Let
$A,B: \cC^\infty_c(D)\rightarrow\mathscr D'(D)$ be continuous operators.
We write 
\begin{equation} \label{e-gue201114yyd}
\mbox{$A\equiv B$ (on $D$)} 
\end{equation}
if $A-B$ is a smoothing operator. We say that $A$ is properly supported if the restrictions of the two projections 
$(x,y)\rightarrow x$, $(x,y)\rightarrow y$ to ${\rm Supp\,}K_A$
are proper. 

For $m\in\mathbb R$, let $H^m(D)$ denote the Sobolev space
of order $m$ on $D$. Put
\begin{gather*}
H^m_{\rm loc\,}(D)=\big\{u\in\mathscr D'(D);\, \varphi u\in H^m(D),
      \, \forall\varphi\in \cC^\infty_c(D)\big\}\,,\\
      H^m_{\rm comp\,}(D)=H^m_{\rm loc}(D)\cap\mathscr E'(D)\,.
\end{gather*}

Let $D$ be an open coordinate patch of $X$ with local coordinates $x$. We recall the following H\"ormander symbol space

\begin{defn}\label{d-gue201114yyd}
For $m\in\mathbb R$, $S^m_{1,0}(D\times D\times\mathbb{R}_+)$ 
is the space of all $a(x,y,t)\in\cC^\infty(D\times D\times\mathbb{R}_+)$ 
such that for all compact $K\Subset D\times D$ and all $\alpha, 
\beta\in\mathbb N^{2n-1}_0$, $\gamma\in\mathbb N_0$, 
there is a constant $C_{\alpha,\beta,\gamma}>0$ such that 
\[|\partial^\alpha_x\partial^\beta_y\partial^\gamma_t a(x,y,t)|\leq 
C_{\alpha,\beta,\gamma}(1+|t|)^{m-|\gamma|},\ \ 
\mbox{for all $(x,y,t)\in K\times\mathbb R_+$, $t\geq1$}.\]
Put 
\[
S^{-\infty}(D\times D\times\mathbb{R}_+)
:=\bigcap_{m\in\mathbb R}S^m_{1,0}(D\times D\times\mathbb{R}_+).\]
Let $a_j\in S^{m_j}_{1,0}(D\times D\times\mathbb{R}_+)$, 
$j=0,1,2,\ldots$ with $m_j\rightarrow-\infty$, $j\rightarrow\infty$. 
Then there exists $a\in S^{m_0}_{1,0}(D\times D\times\mathbb{R}_+)$ 
unique modulo $S^{-\infty}$, such that 
$a-\sum^{k-1}_{j=0}a_j\in S^{m_k}_{1,0}(D\times D\times\mathbb{R}_+)$ 
for $k=1,2,\ldots$. 

If $a$ and $a_j$ have the properties above, we write $a\sim\sum^{\infty}_{j=0}a_j$ in 
$S^{m_0}_{1,0}(D\times D\times\mathbb{R}_+)$. We write
\begin{equation}  \label{e-gue201114yydI}
s(x, y, t)\in S^{m}_{{\rm cl\,}}(D\times D\times\mathbb{R}_+)
\end{equation}
if $s(x, y, t)\in S^{m}_{1,0}(D\times D\times\mathbb{R}_+)$ and 
\begin{equation}\label{e-gue201114yydII}
\begin{split}
&s(x, y, t)\sim\sum^\infty_{j=0}s_j(x, y)t^{m-j}\text{ in }S^{m}_{1, 0}
(D\times D\times\mathbb{R}_+)\,,\\
&s_j(x, y)\in \cC^\infty(D\times D),\ j\in\N_0.
\end{split}\end{equation}
\end{defn}

Let $X$ be an orientable paracompact smooth manifold of dimension $2n+1$ with $n\geq 1$. Let $T^{1,0}X$ be a subbundle of rank $n$ of the complexified tangent bundle $\C TX$. Let $T^{0,1}X:=\ov {T^{1,0}X}$. Let $\cC^\infty(X,T^{1,0}X)$ be the space of smooth sections of $T^{1,0}X$ on $X$. $T^{1,0}X$ is called a CR structure of $X$, if 
\be
T^{1,0}X\cap T^{0,1}X=\{0\}, \quad [\cC^\infty(X,T^{1,0}X),\cC^\infty(X,T^{1,0}X)]\subset \cC^\infty(X,T^{1,0}X). 
\ee
$(X,T^{1,0}X)$ is called a CR manifold of dimension $2n+1$, if $T^{1,0}X$ is a CR structure of $X$. 

Let $\langle\cdot|\cdot\rangle: \C TX\times \C TX\rightarrow \C$ be a smooth Hermitian inner product on $\C TX$ such that
\be\label{*}
\langle u|v \rangle &\in& \R\quad\mbox{for all}~u,v\in TX,\\\label{*'}
\langle \alpha|\beta\rangle&=&0\quad\mbox{for all }~\alpha\in T^{1,0}X,~\beta\in T^{0,1}X.
\ee
The Hermitian norm is given by $|\cdot|:=\sqrt{\langle \cdot|\cdot \rangle}$. 
We call $\langle\,\cdot\,|\,\cdot\,\rangle$ a Hermitian metric on $X$. Let $g^{TX}$ be the Riemannian metric 
on $TX$ given by $g^{TX}(u,v):=\langle u|v\rangle$, $u,v\in TX$. 
For $p, q\in\mathbb N_0$, define $T^{p,q}X:=(\Lambda^pT^{1,0}X)\wedge(\Lambda^qT^{0,1}X)$ and let 
$T^{\bullet,\bullet}X=\bigoplus_{p,q\in \N_0}T^{p,q}X$. For $u\in \C TX$ and $\phi \in \C T^*X$, the pointwise duality is defined by $\langle u,\phi \rangle:=\phi(u).$
Let $T^{*1,0}X\subset \C T^*X$ be the dual bundle of $T^{1,0}X$ and $T^{*0,1}X\subset \C T^*X$ be the dual bundle of $T^{0,1}X$.
For $p, q\in\mathbb N_0$, the bundle of $(p,q)$ forms is denoted by $T^{*p,q}X:=(\Lambda^pT^{*1,0}X)\wedge(\Lambda^qT^{*0,1}X)$ and let 
$T^{*\bullet,\bullet}X:=\oplus_{p,q\in\mathbb N_0}T^{*p,q}X$. The induced Hermitian inner product  on $T^{\bullet,\bullet}X$ and $T^{*\bullet,\bullet}X$ are still denoted by $\langle\cdot|\cdot\rangle$. The Hermitian norms are still denoted by $|\cdot|$. Let $\Omega^{p,q}(X):=\cC^{\infty}(X,T^{*p,q}X)$ be the space of smooth $(p,q)$-forms on $X$ and $\Omega^{\bullet,\bullet}(X):=\bigoplus_{p,q\in \N_0}\Omega^{p,q}(X)$. 
Let $\cC^\infty(X):=\Omega^{0,0}(X)$. 

Let $Y\subset X$ be a subset. 
We denote by $T^{1,0}Y:=T^{1,0}X|_Y$ and $T^{0,1}Y:=T^{0,1}X|_Y$ the induced bundles over $Y$.  
    
    Since $X$ is orientable, there exists a nowhere vanishing $(2n+1)$-form $dv_X$ on $X$ such that $|dv_X|=1$. 
    Let $D\subset X$ be an open subset. 
Let $\{ \ov L_j \}_{j=1}^n
$ be an orthonormal frame of
$T^{1,0}D $ with the dual (orthonormal) frame $\{ \ov e_j \}_{j=1}^n$ of $T^{*1,0}D $. Let
$\{  L_j\}_{j=1}^n$ be an orthonormal frame of $T^{0,1}D$ with the dual (orthonormal) frame $\{ e_j \}_{j=1}^n$ of $T^{*0,1}D$. 
    Note that
     $i^n \ov e_1\wedge  e_1\wedge\ldots\wedge \ov e_n\wedge e_n$
    is a real $2n$-form on $D$ and is independent of the choice of orthonormal frame $\{\overline e_j\}^n_{j=1}$ and hence 
    $i^n \ov e_1\wedge  e_1\wedge\ldots\wedge \ov e_n\wedge e_n$ is globally defined. There exists a real $1$-form $\omega_0\in\cC^\infty(X,T^*X)$ on $X$ such that $|\omega_0|=1$, $\omega_0$ is orthogonal to $T^{*1,0}X\oplus T^{*1,0}X$ and
    \be
    {dv_X=(i^n \ov e_1\wedge  e_1\wedge\ldots\wedge \ov e_n\wedge e_n)\wedge \omega_0.}
    \ee
There exists a real vector field $T\in\cC^\infty(X,TX)$ such that $|T|=1$ and
\be
{\langle T,\omega_0\rangle=1,}\quad \langle T,e_j\rangle=\langle T,\ov e_j \rangle=0,\quad j=1,\ldots,n.
\ee
Thus, $\omega_0$ and $T$ are uniquely determined.
$\omega_0$ is called the uniquely determined global real $1$-form.
$T$ is called the uniquely determined global real vector field.
We have the orthogonal decompositions with respect to the Hermitian metric $\Theta_X$:
\be
\C TX&=&T^{1,0}X\bigoplus T^{0,1}X\bigoplus \C \{T\},\\
\C T^*X&=&T^{*1,0}X\bigoplus T^{*0,1}X\bigoplus \C {\{\omega_0\}.}
\ee

\begin{defn}
	For $x\in X$, the Levi form on $T^{1,0}_xX$ is given by, for $u,v\in T_x^{1,0}X$, we have 
	{\be\label{e-201120levi}
	\cL_x(u,\ov v):=\frac{-1}{2i}\left\langle[U,\ov V](x),\omega_0(x)\right\rangle,
	\ee}
	where $U, V\in \cC^\infty(X,T^{1,0}X)$ with $U(x)=u$, $V(x)=v$. 
\end{defn}
Note that $\cL_x$ is independent of the choice of $U$ and $V$, and {$\cL_x(u,\ov v)=\frac{1}{2i}\langle U(x)\wedge \ov V(x),d\omega_0(x) \rangle$} for $x\in X$.

\begin{defn}\label{d-gue201204yyd}
	Let $\pi^{p,q}: \Lambda^{p+q} \C T^*X\longrightarrow T^{*p,q}X$ be the natural projection for $p,q\in \N_0$, $p+q\geq1$. The tangential (resp. anti-tangential) Cauchy-Riemann operator is given by
	\be
	\ddbar_b:=\pi^{p,q+1}\circ d: \Omega^{p,q}(X)\longrightarrow\Omega^{p,q+1}(X),\\
	\dbar_b:=\pi^{p+1,q}\circ d: \Omega^{p,q}(X)\longrightarrow\Omega^{p+1,q}(X).
	\ee
\end{defn}

Let $D\subset X$ be an open set. 
Let $\Omega^{p,q}_c(D)$ be the space of $(p,q)$-forms on $D$ with compact support in $D$. Let $\Omega^{\bullet,\bullet}_c(D):=\bigoplus_{p,q\in \N_0}\Omega^{p,q}_c(D)$. 
We write $\cC^\infty_c(D):=\Omega^{0,0}_c(D)$. Let $(\,\cdot\,|\,\cdot\,)$ be the $L^2$ inner product on $\Omega^{\bullet,\bullet}_c(X)$ induced by $\langle\,\cdot\,|\,\cdot\,\rangle$. Note that 
\be
(u|v):=\int_X\langle u(x)|v(x) \rangle dv_X(x),\ \ u, v\in\Omega^{\bullet,\bullet}_c(X),
\ee
where {$dv_X:=(\Theta_X^n/n!)\wedge\omega_0$} 
is the volume form induced by the Hermitian metric $\Theta_X$ on $X$. 
Let $L^2_{p,q}(X)$ be the completion of  $\Omega^{p,q}_c(X)$ with respect to $(\,\cdot\,|\,\cdot\,)$. 
Let $L^2_{\bullet,\bullet}(X):=\bigoplus_{p,q\in \N_0} L^2_{p,q}(X)$. We write $L^2(X):=L^2_{0,0}(X)$. 
We denote by $\|u\|^2:=(u|u)$ the $L^2$-norm on $X$. Let $\ddbar_b^*$ and $\dbar_b^*$ be the formal adjoints of $\ddbar_b$ and $\dbar_b$ with respect to $(\,\cdot\,|\,\cdot\,)$ respectively. Let $\square_b:=\ddbar_b\ddbar_b^*+\ddbar_b^*
\ddbar_b$ be the Kohn Laplacian on $\Omega^{\bullet,\bullet}(X)$. Let $\ov \square_b:=\dbar_b\dbar_b^*+\dbar_b^*
\dbar_b$ be the anti-Kohn Laplacian on $\Omega^{\bullet,\bullet}(X)$. 
We still denoted by $\ddbar_b$ the maximal extension and by $\ddbar_b^*$ 
the Hilbert space adjoint with respect to the $L^2$-inner product on $X$. 
We also denote by 
\begin{equation}\label{e:gaf1}
\square_b=\ddbar_b\ddbar_b^*+\ddbar_b^*
\ddbar_b: \Dom\square_b\subset 
L^2_{\bullet,\bullet}(X)\rightarrow L^2_{\bullet,\bullet}(X)
\end{equation}
the Gaffney extension of the Kohn Laplacian with the domain 
\begin{equation}\label{e:gaf2}
\Dom\square_b=\big\{u\in L^2_{\bullet,\bullet}(X): 
u\in\Dom\ddbar_b\cap\Dom\ddbar_b^*,
~\ddbar_b u\in\Dom\ddbar_b^*,
~\ddbar_b^* u\in\Dom\ddbar_b\big\}.
\end{equation}
By a result of Gaffney $\square_b$ is a self-adjoint operator
(see e.g.\ \cite[Proposition 3.1.2]{MM}).

\subsection{The $\R$-action on $X$} 
Let $(X,T^{1,0}X)$ be a CR manifold of $\dim X=2n+1$. Let $r: \mathbb R\times X\rightarrow X$, $r(x)=r\circ x$ for $r\in\R$, be a $\R$-action on $X$, see \cite{HHL17}. Let $\oh T$ be the  global real vector field on $X$ given by
\be
(\oh Tu)(x):=\frac{\dbar}{\dbar r}(u(r\circ x))|_{r=0}, \quad u\in \cC^\infty(X).
\ee
    
\begin{defn}\label{d-gue201128yyd}
The $\R$-action is called CR if $[\oh T,\cC^{\infty}(X,T^{1,0}X)]\subset \cC^{\infty}(X,T^{1,0}X)$. The $\R$-action is called transversal if $\C T_xX=T_x^{1,0}X\oplus T_x^{0,1}X\oplus\C \oh T(x)$ at every $x\in X$. The $\R$-action is called locally free, if $\oh T(x)\neq 0$ at every $x\in X$. 
\end{defn}

From now on, assume that $X$ admits a transversal and CR $\R$-action. We take the Hermitian metric $\langle\,\cdot\,|\,\cdot\,\rangle$ so that $\oh T=T$. 

We use the local coordinates of Baouendi-Rothschild-Trevers (BRT charts) \cite[theorem 6.5]{HM16} extensively as follows.

\begin{thm}[BRT charts]
	For each point $x\in X$, there exists a coordinate neighbourhood $D=U\times I$ with coordinates $x=(x_1,\ldots,x_{2n+1})$ centered at $0$, 
	where $U=\{z=(z_1,\ldots,z_n)\in\C^n:|z|<\epsilon \}$ and $I=\{x_{2n+1}\in \R: |x_{2n+1}|<\epsilon_0 \}$, $\epsilon, \epsilon_0>0$, $z=(z_1,\ldots,z_n)$ and $ z_j=x_{2j-1}+\sqrt{-1}x_{2j}, j=1,\ldots,n$, such that 
	\be
	T=\frac{\dbar}{\dbar x_{2n+1}}\quad\mbox{on}\quad D,
	\ee
	and there exists $\phi\in \cC^\infty(U,\R)$ independent of $x_{2n+1}$ satisfying that
	\be
	\left\{Z_j :=\frac{\dbar}{\dbar z_j}+i\frac{\dbar\phi}{\dbar z_j}(z)\frac{\dbar}{\dbar x_{2n+1}}\right\}_{j=1}^n
	\ee
	is a frame of $T^{1,0}D$, and 
	\be
	\left\{	dz_j\right\}_{j=1}^n\subset T^{*1,0}D
	\ee 
	is the dual frame.
\end{thm}

Let $D=U\times I$ be a BRT chart. Let $f\in\cC^\infty(D)$ and $u\in\Omega^{p,q}(D)$ with $u=\sum_{I,J}u_{IJ}dz_I\wedge d\ovz_J$ with ordered sets $I,J$ and $u_{IJ}\in\cC^\infty(D)$, for all $I, J$. We have 
\be
&&{df=\sum_{j=1}^nZ_j(f)dz_j+\sum_{j=1}^n\ov Z_j(f)d\ovz_j+T(f)\omega_0},\\
&&\dbar_b f =\sum_{j=1}^n Z_j(f)dz_j,\quad\ddbar_b f =\sum_{j=1}^n\ov Z_j(f)d\ovz_j,\\
&&\dbar_b u
=\sum_{I,J} (\dbar_b u_{IJ})\wedge dz_I\wedge d\ovz_J,\quad \ddbar_b u
=\sum_{I,J} (\ddbar_b u_{IJ})\wedge dz_I\wedge d\ovz_J.
\ee 
For $u\in \Omega^{p,q}(X)$, let $\mL_T u$ be the Lie derivative of $u$ in the direction of $T$. For simplicity, we write $Tu$ to denote $\mL_T u$. Since the $\R$-action is CR, $Tu\in\Omega^{p,q}(X)$. On a BRT chart $D$, for $u\in\Omega^{p,q}(D)$, $u=\sum_{I,J}u_{IJ}dz_I\wedge d\overline z_J$, 
we have $Tu=\sum_{I,J}(Tu_{IJ})\wedge dz_I\wedge d\ovz_J$ on $D$.


The Levi form $\cL$ in a BRT chart $D\subset X$ has the form 
\be
\cL=\dbar\ddbar\phi|_{T^{1,0}X}. 
\ee  
Indeed, the global real $1$-form is
{$\omega_0=dx_{2n+1}-\sum_{j=1}^n(i\frac{\dbar\phi}{\dbar z_j}dz_j-i\frac{\dbar\phi}{\dbar\ov z_j}d\ov z_j)$} on $D$.
 
 
From now on, we assume that $\Theta_X$ is $\mathbb R$-invariant. Let $D=U\times I$ be a BRT chart. 
The $(1,1)$ form $\Theta=\Theta_U$ on $U$ is defined by, for $x=(z,x_{2n+1})\in D$,
\be
\Theta(z):=\Theta_X(x).
\ee
Note that it is independent of $x_{2n+1}$. More precisely, 
\be
\Theta(z)=\sqrt{-1}\sum^n_{j,k=1}\left\langle Z_j|Z_k \right\rangle(x)dz_j\wedge d\ov z_k.
\ee
Note that for another BRT coordinates $D=\til U\times \til I$, $y=(w,y_{2n+1})$, there exist biholomorphic map $H\in \cC^\infty(U,\tilde U)$ and $G\in\cC^\infty(U,\mathbb R)$ such that $H(z)=w$, for all $z\in U$, $y_{2n+1}=x_{2n+1}+G(z)$, for all $(z,x_{2n+1})\in U\times I$ and $\tilde U=H(U)$, $\tilde I=I+G(U)$. We deduce that $\Theta$ is independent of the choice of BRT coordinates, i.e., $\Theta=\Theta_U=\Theta_{\til U}$. 

Until further notice, we work on a BRT chart $D=U\times I$. For $p, q\in\mathbb N_0$, let $T^{*p,q}U$ be the bundle of $(p,q)$ forms on $U$ and let $T^{*\bullet,\bullet}U:=\oplus_{p,q\in\mathbb N_0}T^{*p,q}U$. 
For $p, q\in\mathbb N_0$, let $T^{p,q}U$ be the bundle of $(p,q)$ vector fields on $U$ and let $T^{\bullet,\bullet}U:=\oplus_{p,q\in\mathbb N_0}T^{p,q}U$.
The  $(1,1)$ form $\Theta$ induces Hermitian metrics  on $T^{\bullet,\bullet}U$ and $T^{*\bullet,\bullet}U$. We shall use 
$\langle\,\cdot\,,\,\cdot\,\rangle_h$ to denote all the induced Hermitian metrics. 
The volume form on $U$ induced by $\Theta$ is given by $d\lambda(z):=\Theta^n/n!$. Thus, the volume form $dv_X$ can be represented by 
\be
dv_X(x)=d\lambda(z)\wedge dx_{2n+1} ~\mbox{on}~D.
\ee
The $L^2$-inner product on $\Omega^{\bullet,\bullet}_c(U)$ with respect to $\Theta$ is given by
\be
\langle s_1,s_2\rangle_{L^2(U)}:=\int_U\langle s_1(z),s_2(z) \rangle_h d\lambda(z),\ \ s_1, s_2\in\Omega^{\bullet,\bullet}_c(U).
\ee
Let $t\in\R$ be fixed. The $L^2$-inner product on $\Omega^{\bullet,\bullet}_c(U)$ with respect to $\Theta$ and $e^{-2t\phi(z)}$ is given by
\be
\langle s_1,s_2\rangle_{L^2(U,e^{-2t\phi})}:=\int_U\langle s_1(z),s_2(z) \rangle_h e^{-2t\phi(z)} d\lambda(z),\ \ s_1, s_2\in\Omega^{\bullet,\bullet}_c(U).
\ee
The Chern curvature of $K_U^*:=\det(T^{1,0}U)$ with respect to $\Theta$ is given by 
\[R^{K_U^*}:=\ddbar\dbar\log\det\left(\langle \frac{\dbar}{\dbar z_j}, \frac{\dbar}{\dbar z_k}\rangle_h\right)^n_{j,k=1},\quad  R^{K_U^*}\in\Omega^{1,1}(U).\]
On $X$, define $K^*_X:=\det(T^{1,0}X)$. Then, $K^*_X$ is a CR line bundle over $X$. The Chern curvature $R^{K^*_X}$ of $K^*_X$ with respect to $\Theta_X$ is defined as   
follows: On a BRT chart $D$, let 
\begin{equation}\label{e-gue201025yyd}
R^{K_X^*}:=\ddbar_b\dbar_b\log\det\left(\langle Z_j|Z_k \rangle\right)_{j,k=1}^n.
\end{equation}
It is easy to see that $R^{K_X^*}$ is independent of the choice of BRT coordinates and hence $R^{K_X^*}$ is globally defined, i.e. $R^{K^*_X}\in\Omega^{1,1}(X)$ (see e.g. \cite[(6.6)]{HM16}). 

Let $\{  L_j\}_{j=1}^n$ be an $\mathbb R$-invariant orthonormal frame of $T^{0,1}D$ with the dual (orthonormal) frame $\{ e_j \}_{j=1}^n$. Then $\{ \ov L_j \}_{j=1}^n
$ is an $\mathbb R$-invariant orthonormal frame of
$T^{1,0}D $ with the dual (orthonormal) frame $\{ \ov e_j \}_{j=1}^n$. Since $\Theta_X$ is $\mathbb R$-invariant, 
there exist $c_j^k=c_j^k(z), w_j^k=w_j^k(z)\in\cC^\infty(U)$, $j,k=1,\ldots,n$, satisfying $\sum_{k=1}^n c_j^kw_k^l=\delta_j^l$, for all $j, l=1,\ldots,n$, such that for $j=1,\ldots,n$,
\be
&&\ov L_j=\sum_{k=1}^n \ov c_j^k  Z_k,\quad\quad \ov e_j=\ov w^j_k dz_k,\\
&&L_j=\sum_{k=1}^n c_j^k \ov Z_k,\quad\quad   e_j= w^j_k d\ov z_k.
\ee
We can check that $\{w_j:=\sum_{k=1}^n \ov c_j^k \frac{\dbar}{\dbar z_k};\, j=1,\ldots,n\}$ and 
$\{\ov w_j:=\sum_{k=1}^n  c_j^k \frac{\dbar}{\dbar \ov z_k};\, j=1,\ldots,n\}$ are orthonormal frames for $T^{1,0}U$ and $T^{0,1}U$ with respect to $\Theta$ respectively and 
$\{\overline e_j;\, j=1,\ldots,n\}$,  $\{e_j;\, j=1,\ldots,n\}$ are dual frames for $\{w_j;\, j=1,\ldots,n\}$ and 
$\{\ov w_j;\, j=1,\ldots,n\}$ respectively. We also write $w^j$ and $\overline w^j$ to denote $\overline e_j$ and $e_j$ respectively, $j=1,\ldots,n$. 

\subsection{The Fourier transform on $D$}
 
Let $D=U\times I$ be a BRT chart. Let $f\in \cC^\infty_c(D)$. We write $f=f(x)=f(z,x_{2n+1})$. For each fixed $x_{2n+1}\in I$, $f(\cdot,x_{2n+1})\in \cC_c^\infty(U)$. For each fixed $z\in U$, $f(z,\cdot)\in \cC_c^\infty(I)$. Let $p,q\in \N_0$, $u\in\Omega_c^{p,q}(D)$. We write $u=\sum_{I,J} u_{IJ}dz_I\wedge d\ov z_J\in \Omega^{p,q}_c(D)$ 
and we always assume that the summation is performed only over increasingly ordered indices $I=i_1<i_2<\ldots<i_p, J=j_1<j_2<\ldots<j_q$, and $u_{IJ}\in\cC^\infty_c(D)$, 
for all $I, J$.  For each fixed $z\in U$, $u_{IJ}(z,\cdot)\in \cC_c^\infty(I)$. 

\begin{defn}
The Fourier transform of the function $f\in\cC^\infty_c(D)$ with respect to $x_{2n+1}$, denoted by $\oh f$, is defined by 
\be
\oh{f}(z,t):=\int_{-\infty}^{\infty}e^{-itx_{2n+1}}f(z,x_{2n+1}) dx_{2n+1}\in\cC^\infty(U\times\mathbb R).
\ee	 
The Fourier transform of the form $u=\sum_{I,J} u_{IJ}dz_I\wedge d\ov z_J\in \Omega^{p,q}_c(D)$ with respect to $x_{2n+1}$, denoted by $\oh u$, is defined by
	\be
	\oh u(z,t)=\sum_{I,J}\oh u_{IJ}(z,t)dz_I\wedge d\ov z_J\in\Omega^{p,q}(U\times\mathbb R):=\cC^\infty(U\times\mathbb R, T^{*p,q}U).
	\ee
\end{defn}
Note that $\oh f\in \cC^\infty(U\times \R)$ and $\oh f(\cdot,t)\in\cC_c^\infty(U)$ for every $t\in\R$. Similary, $\oh u\in\Omega^{p,q}(U\times \R):=\cC^\infty(U\times\R, \wedge^{p,q}T^*U)$ and $\oh u(\cdot,t)\in \Omega^{p,q}_c(U)$ for every $t\in \R$. 
From Parserval formula, we have for $u, v\in \Omega_c^{p,q}(D)$, 
\be
\int_{-\infty}^{\infty} \langle u(z,x_{2n+1})| v(z,x_{2n+1})\rangle dx_{2n+1}=(1/2\pi) \int_{-\infty}^{\infty} \langle\oh u(z,t), \oh v(z,t)\rangle_h dt,
\ee
for every $z\in U$. By using integration by parts, we have for $u\in\Omega^{p,q}_c(D)$, 
\be
-\sqrt{-1}\oh{Tu}=t\oh u,\quad  i.e.,\quad~-\sqrt{-1}\oh{\frac{\dbar u}{\dbar x_{2n+1}}}(z,t)=t\oh{u}(z,t).
\ee
   
Let $t\in \R$ be fixed. Let $|(z,1)|^2_h:=e^{-2t\phi(z)}$ be the Hermitian metric on the trivial line bundle $U\times \C$ over $U$. The Chern connection of $(U\times \C,e^{-2t\phi})$ is given by
\begin{equation}\label{e-gue201025yydII}
\nabla^{(U\times \C,e^{-2t\phi})}=\nabla^{1,0}+\nabla^{0,1},  \quad \nabla^{1,0}=\dbar-2t\dbar\phi, \quad \nabla^{0,1}=\ddbar.
\end{equation}
Indeed, $\nabla^{(U\times \C,e^{-2t\phi})}=d+h^{-1}\dbar h=d+e^{2t\phi}\dbar(e^{-2t\phi})$.
The curvature of  $(U\times \C,e^{-2t\phi})$ is
\be
R^{(U\times\C,e^{-2t\phi})}=\left(\nabla^{(U\times \C,e^{-2t\phi})}\right)^2=2t\dbar\ddbar\phi.
\ee
We can identify $\dbar\ddbar\phi$ with Levi form $\cL$ and write $R^{(U\times\C,e^{-2t\phi})}=2t\cL$. 
Moreover, we will identify $\Omega^{\bullet,\bullet}(U)$ and $\Omega^{\bullet,\bullet}_c(U)$ with 
$\Omega^{\bullet,\bullet}(U, U\times\mathbb C)$ and $\Omega^{\bullet,\bullet}_c(U,U\times\mathbb C)$ respectively. 

\begin{prop}\label{p-gue201025yyd}
	Let $u,v\in \Omega^{\bullet,\bullet}_c(D)$. We have
	\be
	&&\oh{\ddbar_b u}=e^{-t\phi}\ddbar(e^{t\phi}\oh u) \quad\mbox{on}~U\times\R,\\
	&&\oh{\ddbar_b^* v}=e^{-t\phi}\ddbar^*(e^{t\phi}\oh v)\quad\mbox{on}~U\times\R,\\
	&&\oh {\dbar_b  u}=e^{-t\phi}\nabla^{1,0}(e^{t\phi}\oh u) \quad\mbox{on}~U\times\R,\\
	&&\oh{\dbar_b^*u}=e^{-t\phi}\nabla^{1,0*}(e^{t\phi}\oh u)\quad\mbox{on}~U\times\R,
	\ee
	where $\ddbar^*,\nabla^{1,0*}$ are the formal adjoints of $\ddbar,\nabla^{1,0}$ with respect to $\langle\,\cdot\,,\,\cdot\,\rangle_{L^2(U,e^{-2t\phi})}$ respectively and $\ddbar_b^*$, $\dbar^*_b$ are the formal adjoints of $\ddbar_b$, $\dbar_b$ with respect to $(\,\cdot\,|\,\cdot\,)$ respectively. 
	\end{prop} 

\begin{proof}
Let $u=\sum_{I,J}u_{IJ}dz_I\wedge d\ovz_J$.	
By
$\ddbar_b u=\sum_{I,J} \sum_{j=1}^n\left(\frac{\dbar u_{IJ}}{\dbar \ovz_j}-i\frac{\dbar\phi}{\dbar \ovz_j}\frac{\dbar u_{IJ}}{\dbar x_{2n+1}}\right) d\ovz_j\wedge dz_I\wedge d\ovz_J$,
\be
\begin{split}
\left(\oh{\ddbar_b u}\right)(z,t)
&=\sum_{I,J} \sum_{j=1}^n\left(\frac{\dbar \oh u_{IJ}}{\dbar \ovz_j}(z,t)+t\frac{\dbar\phi}{\dbar \ovz_j}(z)\oh u_{IJ}(z,t)\right) d\ovz_j\wedge dz_I\wedge d\ovz_J\\
&=e^{-t\phi(z)}\ddbar(e^{t\phi} \sum_{I,J}\oh u_{IJ}dz_I\wedge d\ovz_{J})(z,t)\\
&=\left(e^{-t\phi}\ddbar(e^{t\phi}\oh u)\right)(z,t).
\end{split}
\ee
Thus the first equality holds. From Parserval's formula, 
\be
\begin{split}
	&(\ddbar_b u|v)\\
	=&\int_D \langle\ddbar_b u|v\rangle d\lambda(z)dx_{2n+1}\\
	=&\int_{U}\left((2\pi)^{-1}\int_{-\infty}^{\infty}\langle \oh {\ddbar_b u}(z,t),\oh v(z,t)\rangle_h dt\right) d\lambda(z)\\
	=&(2\pi)^{-1}\int_{-\infty}^{\infty}\int_U\langle e^{-t\phi}\ddbar(e^{t\phi}\oh u), \oh v\rangle_h d\lambda(z)dt\\
	=&(2\pi)^{-1}\int_{-\infty}^{\infty}  \langle\ddbar(e^{t\phi}\oh u),e^{t\phi}\oh v \rangle_{L^2(U,e^{-2t\phi})} dt\\
	=&(2\pi)^{-1}\int_{-\infty}^{\infty}  \langle e^{t\phi}\oh u ,\ddbar^*(e^{t\phi}\oh v) \rangle _{L^2(U,e^{-2t\phi})} dt\\
	=&(2\pi)^{-1}\int_{-\infty}^{\infty}\int_U \langle \oh u,e^{-t\phi}\ddbar^*(e^{t\phi}\oh v) \rangle_h d\lambda(z)dt. 
\end{split}
\ee
Meanwhile, we have  
\be
(\ddbar_b u|v)=(u|\ddbar_b^* v)=(2\pi)^{-1}\int_{-\infty}^{\infty}\int_{U}\langle\oh u,\oh{\ddbar_b^* v}\rangle_h d\lambda(z)dt.
\ee
Thus the second equality holds. The proofs of the third and the fourth equalities are similar. 
\end{proof}       
      
\section{CR Bochner formula with $\R$-action}

Recall that we work with the assumption that $X$ admits a transversal 
CR $\R$-action on $X$. {We will prove Bochner-Kodaira-Nakano 
formulas in the CR setting. 
They are refinements of Tanaka's basic identities \cite[Theorems 5.1, 5.2]{Tan75} 
in our context. Namely, Tanaka's formulas hold for any strictly pseudoconvex
manifold endowed with the Levi metric, while our formulas are specific
to CR manifolds with $\R$-action endowed with arbitrary Hermitian metric
$\Theta_X$.}

\subsection{CR Bochner-Kodaira-Nakano formula I} 
Analogue to \cite[(1.4.32)]{MM}, we define the Lefschetz operator $\Theta_X\wedge\cdot$ on $\bigwedge^{\bullet,\bullet}(T^*X)$ and its adjoint $\Lambda=i(\Theta_X)$ with respect to the Hermitian inner product $\langle\cdot|\cdot\rangle$ associated with $\Theta_X$. The Hermitian torsion of $\Theta_X$ is defined by  
\be
\mT:=[\Lambda,\dbar_b\Theta_X].
\ee
Let $D=U\times I$ be a BRT chart and let $\{\overline L_j\}^n_{j=1}\subset T^{1,0}D$,  $\{\overline e_j\}^n_{j=1}\subset T^{*1,0}D$, 
$\{w_j\}^n_{j=1}\subset T^{1,0}U$ be as in the discussion after \eqref{e-gue201025yyd}. We can check that 
\be
\Theta_X\wedge\cdot=\sqrt{-1}\ov e_j\wedge e_j\wedge\cdot, \quad \Lambda=-\sqrt{-1}i_{L_j}i_{\ov L_j}\quad \mbox{on}~D.
\ee
Note that $i_{L_j}$ and $i_{\ov L_j}$ are the adjoints of $e_j\wedge$ and $\ov e_j\wedge$ respectively.

Since $\dbar\Theta(z)=\dbar_b\Theta_X(x)$ on $D$, and
$\Theta\wedge\cdot=\sqrt{-1}\ov e_j\wedge e_j\wedge\cdot $, $\Lambda=-\sqrt{-1}i_{\ov w_j}i_{w_j}$ on $U$, see \cite[1.4.32]{MM}, we have $\mT=[\Lambda,\dbar_b\Theta]=[\Lambda,\dbar \Theta]$ on $\Omega^{\bullet,\bullet}(D)$, which is independent of $x_{2n+1}$.
We remark that $\mT$ is a differential operator of order zero. With respect to the Hermitian inner product $\langle\cdot|\cdot\rangle$ associated with $\Theta_X$, we have the adjoint operator $\mT^*$, the conjugate operator $\ov \mT$ and the adjoint of the conjugate operator $\ov \mT^*$ for $\mT$.  

\begin{thm}\label{t-gue201028yyd}
With the notations used above, we have on $\Omega^{\bullet,\bullet}(X)$, 
	\begin{equation}\label{e-gue201025yydI}
	\square_b=\ov \square_b+[2\sqrt{-1}\cL,\Lambda](-\sqrt{-1}T)+(\partial_b\,\mT^*+\mT^*\,\partial_b)-(\ddbar_b\,\ov\mT^*+\ov\mT^*\,\ddbar_b).
	\end{equation}
\end{thm} 

\begin{proof}  
	\comment{Let $\{ B_k\subset X \}_{k\geq 1}$ and $\{ \chi_k\in\cC_c^\infty(B_k,\R) \}_{k\geq 1}$ be a partition of unity with $\sum_{k\geq 1} \chi_k=1$ and $\cup_{k\geq 1} B_k=X$ such that if $B_i\cap B_j\neq\emptyset$ there exists a BTR chart $D=U\times I\subset X$ satisfying $B_i\cup B_j\subset D$. Let $u\in \Omega_c^{p,q}(X)$. Thus $\chi_k u\in \Omega^{p,q}_c(B_k)$ for each $k$. For fixed $t\in\R$, we set
	\be
	s_k(z):=e^{t\phi(z)}\oh{\chi_k u}(z,t)
	\ee
	for each $k$. We have that $s_k(z)\in \Omega_c^{p,q}(U)$ for some BRT chart $D=U\times I$.
	
	In the following, we consider $\chi_i u$ and $\chi_j u$ and suppose $B_i\cap B_j\neq \emptyset$ and $B_i\cup B_j\subset D$ for some BRT chart $D=U\times I$. We apply Bochner-Kodaira-Nakano formula \cite[(1.4.44)]{MM} on $s_i$ and $s_j$  for the trivial line bundle $U\times \C$ endowed with Hermitian metric $|(z,1)|_h^2=e^{-2t\phi(z)}$ over $U$,}
	Since the both side of \eqref{e-gue201025yydI} are globally defined, we can check \eqref{e-gue201025yydI} on a BRT chart. Now, we work on a BRT chart $D=U\times I$. We will use the same notations as before. Let 
	\[\begin{split}
	&\square^{(U\times \C, e^{-2t\phi})}=\ddbar\,\ddbar^*+\ddbar^*\,\ddbar: \Omega^{\bullet,\bullet}_c(U)\rightarrow\Omega^{\bullet,\bullet}_c(U),\\
	&\ov\square^{(U\times \C, e^{-2t\phi})}:=\nabla^{1,0*}\nabla^{1,0}+\nabla^{1,0}\nabla^{1,0*}:  \Omega^{\bullet,\bullet}_c(U)\rightarrow\Omega^{\bullet,\bullet}_c(U),
	\end{split}\]
	where $\nabla^{1,0}$ is given by \eqref{e-gue201025yydII}, $\ddbar^*,\nabla^{1,0*}$ are the formal adjoints of $\ddbar,\nabla^{1,0}$ with respect to $\langle\,\cdot\,,\,\cdot\,\rangle_{L^2(U,e^{-2t\phi})}$ respectively. From~\cite[(1.4.44)]{MM}, 
	
	\be\nonumber
	\square^{(U\times \C, e^{-2t\phi})}=\ov\square^{(U\times \C, e^{-2t\phi})}+[2\sqrt{-1}t\cL,\Lambda]+(\nabla^{1,0}\mT^*+\mT^*\nabla^{1,0})-(\ddbar\,\ov\mT^*+
	\ov\mT^*\,\ddbar).
	\ee
	Let $u, v\in\Omega^{\bullet,\bullet}_c(D)$. Let $s_1(z):=e^{t\phi(z)}\oh{u}(z,t)\in\Omega^{\bullet,\bullet}(U\times\mathbb R)$, 
	$s_2(z):=e^{t\phi(z)}\oh{v}(z,t)\in\Omega^{\bullet,\bullet}(U\times\mathbb R)$.
	Firstly, we have
	\begin{equation}\label{e-gue201025yydh}
	(1/2\pi)\int_{-\infty}^{\infty}\langle\square^{(U\times \C, e^{-2t\phi})}s_1,s_2\rangle_{L^2(U,e^{-2t\phi})}dt=(\,\ddbar_b u\,|\,\ddbar_bv\,)+(\,\ddbar_b^* u\,|\,\ddbar_b^*v\,).
	\end{equation}
	In fact, from Proposition~\ref{p-gue201025yyd}, 
	\begin{equation*}
	\begin{split}
	&\int_{-\infty}^{\infty}\langle\square^{(U\times \C, e^{-2t\phi})}s_1,s_2\rangle_{L^2(U,e^{-2t\phi})}dt\\
	&=\int_{-\infty}^{\infty}\Bigr(\langle\ddbar s_1,\ddbar s_2\rangle_{L^2(U,e^{-2t\phi})}+\langle\ddbar^* s_1,\ddbar^* s_2\rangle_{L^2(U,e^{-2t\phi})}\Bigr)dt\\
	&=\int_{-\infty}^{\infty}\Bigr(\langle \oh{ \ddbar_bu},\oh {\ddbar_bv}  \rangle_{L^2(U)}+\langle \oh{ \ddbar_b^*u},\oh {\ddbar_b^*v}  \rangle_{L^2(U)}\Bigr)dt\\
	&=2\pi (\ddbar_bu|\ddbar_bv)+2\pi (\ddbar_b^*u|\ddbar_b^*v).
	\end{split}
	\end{equation*}
Similarly, we have 
	\begin{equation}\label{e-gue201028yydII}
	(1/2\pi)\int_{-\infty}^{\infty}\langle\ov \square^{(U\times \C, e^{-2t\phi})}s_1,s_2\rangle_{L^2(U,e^{-2t\phi})}dt= (\dbar_bu|\dbar_bv)+(\dbar_b^*u|\dbar_b^*v).
	\end{equation}
	\comment{In fact, it follows from  
	\be
&&\int_{-\infty}^{\infty}\langle
\ov\square^{(U\times \C, e^{-2t\phi})}s_i,s_j\rangle_{L^2(U,e^{-2t\phi})}dt\\
&=&\int_{-\infty}^{\infty}\langle\nabla^{1,0} s_i,\nabla^{1,0} 
s_j\rangle_{L^2(U,e^{-2t\phi})}+
\langle\nabla^{1,0*} s_i,\nabla^{1,0*} s_j\rangle_{L^2(U,e^{-2t\phi})}dt\\
&=&\int_{-\infty}^{\infty}\langle \oh{ \dbar_b (\chi_i u)},
\oh {\dbar_b(\chi_j u)}  \rangle_{L^2(U)}+\langle \oh{ \dbar_b^* (\chi_i u)},\oh {\dbar_b^*(\chi_j u)}  
\rangle_{L^2(U)}dt\\
&=&2\pi (\dbar_b(\chi_i u)|\dbar_b(\chi_j u))+
2\pi (\dbar_b^*(\chi_i u)|\dbar_b^*(\chi_j u)).
\ee}	
Thirdly, we have
	\begin{equation}\label{e-gue201028yydIII}
	(1/2\pi)\int_{-\infty}^{\infty}t\langle [2\sqrt{-1}\cL,\Lambda]s_1,s_2\rangle_{L^2(U,e^{-2t\phi})}dt=([2\sqrt{-1}\cL,\Lambda](-\sqrt{-1}T)u| v).
	\end{equation}
	In fact, it follows from
	\begin{equation*}
	\begin{split}
	&\int_{-\infty}^{\infty}t\langle [2\sqrt{-1}\cL,\Lambda]s_1,s_2\rangle_{L^2(U,e^{-2t\phi})}dt\\
	&=\int_{-\infty}^{\infty}\langle t[2\sqrt{-1}\cL,\Lambda]\oh{u},\oh {v} \rangle_{L^2(U)}dt\\
	&=\int_{-\infty}^{\infty}\langle [2\sqrt{-1}\cL,\Lambda](-\oh{\sqrt{-1}Tu}),\oh{v}\rangle_{L^2(U)}dt\\
	&=2\pi ([2\sqrt{-1}\cL,\Lambda]({-\sqrt{-1}Tu })|{v}).
	\end{split}
	\end{equation*}
	
	Fourthly, we consider the rest terms
	\be
	\langle (\nabla^{1,0}\mT^*+\mT^*\nabla^{1,0}) s_1,s_2 \rangle_{L^2(U,e^{-2t\phi})},\quad
	\langle (\nabla^{0,1}\ov \mT^*+\ov \mT^*\nabla^{0,1}) s_1, s_2 \rangle_{L^2(U,e^{-2t\phi})}.
	\ee	
	By Proposition~\ref{p-gue201025yyd}, we have
	\be
	\begin{split}
	&\int_{-\infty}^{\infty}\langle  (\nabla^{1,0}\mT^*+\mT^*\nabla^{1,0})s_1,s_2\rangle_{L^2(U,e^{-2t\phi})}dt\\
	&=\int_{-\infty}^{\infty}\langle\Bigr(\nabla^{1,0}\mT^*s_1,s_2\rangle_{L^2(U,-2t\phi)}+\langle\mT^*\nabla^{1,0}s_1,s_2\rangle_{L^2(U,-2t\phi)}\Bigr)dt\\
	&=\int_{-\infty}^{\infty}\langle\Bigr(\nabla^{1,0}\mT^* e^{t\phi}\oh {u},e^{t\phi}\oh {v}\rangle_{L^2(U,-2t\phi)}+\langle\mT^*\nabla^{1,0}e^{t\phi}\oh {u},e^{t\phi}\oh {v}\rangle_{L^2(U,-2t\phi)}\Bigr)dt\\
	&=\int_{-\infty}^{\infty}\langle\Bigr(\mT^* e^{t\phi}\oh {u},\nabla^{1,0*}(e^{t\phi}\oh {v})\rangle_{L^2(U,-2t\phi)}+\langle\nabla^{1,0}(e^{t\phi}\oh {u}),\mT e^{t\phi}\oh {v}\rangle_{L^2(U,-2t\phi)}\Bigr)dt\\
	&=\int_{-\infty}^{\infty}\langle\Bigr(\mT^*\oh {u},e^{-t\phi}\nabla^{1,0*}(e^{t\phi}\oh {v})\rangle_{L^2(U)}+\langle e^{-t\phi}\nabla^{1,0}e^{t\phi}\oh {u},\mT \oh {v}\rangle_{L^2(U)}\Bigr)dt\\
	&=\int_{-\infty}^{\infty}\langle\Bigr(\mT^*\oh {u},\oh{\dbar_b^*v} \rangle_{L^2(U)}+\langle \oh{\dbar_bu}, \mT\oh{v} \rangle_{L^2(U)}\Bigr)dt\\
	&=2\pi(\mT^*u|\dbar_b^*v)+2\pi(\dbar_b u|\mT v).
		\end{split}
	\ee
	Thus we obtain
	\begin{equation}\label{e-gue201028yydIV}
	\begin{split}
	&(1/2\pi)\int_{-\infty}^{\infty}\langle(\nabla^{1,0}\mT^*+
	\mT^*\nabla^{1,0})s_1,s_2 \rangle_{L^2(U,e^{-2t\phi})}dt\\
	&=(\mT^* u|\dbar_b^* v)+(\dbar_b u|\mT v)\\
	&=((\dbar_b\mT^*+\mT^*\dbar_b)u|v).
	\end{split}
	\end{equation}	
	Similarly, we obtain
	 \begin{equation}\label{e-gue201028yydV}
	(1/2\pi)\int_{-\infty}^{\infty}\langle(\nabla^{0,1}\ov \mT^*+
	\ov\mT^*\nabla^{0,1})s_1,s_2 \rangle_{L^2(U,e^{-2t\phi})}dt=
	((\ddbar_b\ov{\mT}^*+\ddbar_b\ov{\mT}^*)u|v).
	\end{equation}	
From \eqref{e-gue201025yydh}, \eqref{e-gue201028yydII}, 
\eqref{e-gue201028yydIII}, \eqref{e-gue201028yydIV} and 
\eqref{e-gue201028yydV}, we get that for $u, v\in\Omega^{\bullet,\bullet}_c(D)$, 
	\be\\\nonumber
	(\square_b u|v)=((\ov \square_b+[2\sqrt{-1}\cL,\Lambda](-\sqrt{-1}T)+
	(\dbar_b\mT^*+\mT^*\dbar_b)-(\ddbar_b\ov \mT ^*+\ov\mT^*\ddbar_b))u|v).
	\ee
	The theorem follows. 
\end{proof}

\begin{cor}[CR Nakano's inequality I]\label{c-gue201028yyd}
	With the notations used above, 	 for any $u\in  \Omega_c^{\bullet,\bullet}(X)$,  
	\be    
	\begin{split}
	\frac{3}{2}(\square_b u|u)&\geq ([2\sqrt{-1}\cL,\Lambda](-\sqrt{-1}Tu)|u)\\
	&-\frac{1}{2}(\|\mT u\|^2+\|\mT^* u\|^2+\|\ov \mT u\|^2+\|\ov \mT^*u\|^2).
	\end{split}
	\ee
	If $(X,T^{1,0}X)$ is K\"{a}hler, i.e., $d\Theta_X=0$, then 
	\be
	(\square_b u|u)\geq ([2\sqrt{-1}\cL,\Lambda](-\sqrt{-1}Tu)|u).
	\ee
\end{cor}  
\begin{proof}
	By Cauchy–Schwarz inequality, Theorem~\ref{t-gue201028yyd} and note that $\mT=0$, $\mT^*=0$ if $d\Theta_X=0$, we get the corollary. 
	\end{proof}

The following follows from straightforward calculation, we omit the proof 

\begin{prop}
	For a real $(1, 1)$-form $\sqrt{-1}\alpha\in \Omega^{1,1}(D)$, if we choose local orthonormal frame $\{\overline L_j\}_{j=1}^n$ of $T^{1,0}D$ with the dual frame $\{\overline e_j\}_{j=1}^n$ of $T^{*1,0}D$ such that $\sqrt{-1}\alpha=\sqrt{-1}\lambda_j(x)\overline e_j\wedge e_j$
	at a given point $x\in D$, then for any $f=\sum_{I,J}f_{IJ}(x)\overline e^I \wedge e^J \in \Omega^{\bullet,\bullet}(D)$, we
	have
	\begin{equation}\label{e-gue201028ycd}
	[\sqrt{-1}\alpha,\Lambda]f(x)=\sum_{I,J}\left(\sum_{j\in I} \lambda_j(x)+\sum_{j\in J}\lambda_j(x)-\sum_{j=1}^n\lambda_j(x)\right)f_{IJ}(x)\overline e^I \wedge e^J.
	\end{equation}
\end{prop}
   
\begin{cor}\label{vanish_nq_1}
	With the notations used above, let $\Theta_X$ be a Hermitian metric on $X$ such that	
	\be
	2\sqrt{-1}\cL=\Theta_X. 
	\ee 
	Then for any $u\in \Omega_c^{n,q}(X)$ with $1\leq q\leq n$,
	\be
	\left(-\sqrt{-1}Tu|u\right)\leq \frac{1}{q}\left(\|\ddbar_b u\|^2+\|\ddbar^*_b u\|^2\right).
	\ee
\end{cor} 
\begin{proof}  
	By applying \eqref{e-gue201028ycd} for $\sqrt{-1}\alpha:=2\sqrt{-1}\cL$, $\lambda_j=1$ for all $j$, we have 
	\be
	[2\sqrt{-1}\cL,\Lambda](-\sqrt{-1}T u)=q(-\sqrt{-1}Tu),\ \ \mbox{for all $u\in\Omega^{n,q}_c(X)$}. 
	\ee
	 By $d\Theta_X=d(2\sqrt{-1}\cL)=0$ and Corollary~\ref{c-gue201028yyd}, we obtain
	\be
	(\square_b u|u)\geq ([2\sqrt{-1}\cL,\Lambda](-\sqrt{-1}Tu)|u)=q(-\sqrt{-1}Tu|u).
	\ee
\end{proof} 

Let $E$ be a CR line bundle over $X$ (see Definition 2.4 in~\cite{HHL17}. 
We say that $E$ is a $\mathbb R$-equivariant CR line bundle over $X$ if the $\mathbb R$-action on $X$ can be CR lifted to $E$ and for every point $x\in X$, we can find a $T$-invariant local CR trivializing section of $E$ defined near $x$ (see Definition 2.9 and Definition 2.6 in~\cite{HML17}). Here we also use $T$ to denote the vector field acting on sections of $E$ induced by the $\mathbb R$-action on $E$. From  now on, we assume that $E$ is a $\mathbb R$-equivariant CR line bundle over $X$ with a $\mathbb R$-invariant Hermitian metric $h^E$ on $E$. For $p, q\in\mathbb N_0$, let $\Omega^{p,q}(X,E)$ be the space of smooth $(p,q)$ forms of $X$ with values in $E$ and 
let $\Omega^{\bullet,\bullet}(X,E):=\oplus_{p,q\in\mathbb N_0}\Omega^{p,q}(X,E)$. Let $\Omega^{p,q}_c(X,E)$ be the subspace of $\Omega^{p,q}(X,E)$ whose elements have compact support in $X$ and let $\Omega^{\bullet,\bullet}_c(X,E):=\oplus_{p,q\in\mathbb N_0}\Omega^{p,q}_c(X,E)$. For $p, q\in\mathbb N_0$, let 
\[\ddbar_{b,E}: \Omega^{p,q}(X,E)\rightarrow\Omega^{p,q+1}(X,E)\]
be the tangential Cauchy-Riemann operator with values in $E$. Let $(\,\cdot\,|\,\cdot\,)_E$ be the $L^2$ inner product on $\Omega^{\bullet,\bullet}_c(X,E)$ induced by 
$\langle\,\cdot\,|\,\cdot\,\rangle$ and $h^E$. Let 
\[\ddbar^*_{b,E}: \Omega^{p,q+1}(X,E)\rightarrow\Omega^{p,q}(X,E)\]
be the formal adjoint of $\ddbar_{b,E}$ with respect to $(\,\cdot\,|\,\cdot\,)_E$. Put 
\[\Box_{b,E}:=\ddbar_{b,E}\,\ddbar^*_{b,E}+\ddbar^*_{b,E}\,\ddbar_{b,E}:\Omega^{\bullet,\bullet}(X,E)\rightarrow\Omega^{\bullet,\bullet}(X,E).\]
Let 
\begin{equation}\label{e-gue201031yyd}
\nabla^E: \Omega^{\bullet,\bullet}(X,E)\rightarrow\Omega^{\bullet,\bullet}(X,E\otimes\mathbb CT^*X)
\end{equation}
be the connection on $E$ induced by $h^E$ given as follows: Let $s$ be a $T$-invariant local CR trivializing section of $E$ on an open set $D$ of $X$, 
\begin{equation}\label{e-gue201031ycd}
\abs{s}^2_{h^E}=e^{-2\Phi},\ \ \Phi\in\cC^\infty(D,\mathbb R). 
\end{equation}
Then, 
{\begin{equation}\label{e-gue201031yydI}
\nabla^E(u\otimes s):=(\ddbar_bu+\dbar_bu-2(\dbar\Phi)\wedge u+\omega_0\wedge(Tu))\otimes s, \ \ u\in\Omega^{\bullet,\bullet}(D). 
\end{equation}}
It is straightforward to check that \eqref{e-gue201031yydI} is independent of the choices of $T$-invariant local CR trivializing sections $s$ and hence is globally defined. 
Put 
\begin{equation}\label{e-gue201031yydII}
(\nabla^E)^{0,1}:=\ddbar_b,\ \ (\nabla^E)^{1,0}:=\dbar_b-2\dbar_b\Phi. 
\end{equation}
Let 
\begin{equation}\label{e-gue201031yydIII}
\overline\Box_{b,E}:=(\nabla^E)^{1,0}((\nabla^E)^{1,0})^*+((\nabla^E)^{1,0})^*(\nabla^E)^{1,0}: \Omega^{\bullet,\bullet}(X,E)\rightarrow\Omega^{\bullet,\bullet}(X,E).
\end{equation}
Let $R^E\in\Omega^{1,1}(X)$ be the curvature of $E$ induced by $h^E$ given by $R^E:=-2\ddbar_b\dbar_b\Phi$ on $D$, where $\Phi$ is as in \eqref{e-gue201031ycd}. 
Let $D=U\times I$ be a BRT chart. Since $E$ is $\mathbb R$-equivariant, on $D$, $E$ is a holomorphic line bundle over $U$. We can repeat the proof of Theorem~\ref{t-gue201028yyd} with minor changes and conclude 

\begin{thm}\label{t-gue201031yyd}
Let $E$ be a $\mathbb R$-equivariant CR line bundle over $X$ with a $\mathbb R$-invariant Hermitian metric $h^E$. With the notations used above, we have on $\Omega^{\bullet,\bullet}(X,E)$, 
\begin{equation}\label{e-gue201031ycdi}
\Box_{b,E}=\overline\Box_{b,E}+[2\sqrt{-1}\cL,\Lambda](-\sqrt{-1}T)+[\sqrt{-1}R^E, \Lambda]+\Bigr((\nabla^E)^{1,0}\mT^*+\mT^*(\nabla^E)^{1,0}\Bigr)-\Bigr(\ddbar_{b,E}\overline\mT^*+\overline\mT^*\ddbar_{b,E}\Bigr),
\end{equation}
where $R^E\in\Omega^{1,1}(X)$ is the curvature of $E$ induced by $h^E$. 
\end{thm}
  
\subsection{CR Bochner-Kodaira–Nakano formula II} 

The bundle $K^*_X:=\det(T^{1,0}X)$ is a $\mathbb R$-equivariant CR line bundle over $X$. The $(1,1)$ form $\Theta_X$ induces a $\mathbb R$-invariant Hermitian metric 
$h^{K^*_X}$ on $K^*_X$. Let $R^{K^*_X}$ be the curvature of $K^*_X$ induced by $h^{K^*_X}$. Let \[\Psi: T^{*0,q}X\rightarrow T^{*n,q}X\otimes K^*_X\]
be the natural isometry defined as follows: Let $D=U\times I$ be a BRT chart. Let $\{\overline L_j\}^n_{j=1}\subset T^{1,0}D$, $\{\overline e_j\}^n_{j=1}\subset T^{*1,0}D$ be as in the discussion after \eqref{e-gue201025yyd}. Then, 
\[\Psi u:=\overline e_1\wedge\ldots\wedge\overline e_n\wedge u\otimes(\overline L_1\wedge\ldots\wedge\overline L_n)\in T^{*n,q}X\otimes K^*_X,\ \ u\in T^{*0,q}X.\]
It is easy to see that the definition above is independent of the choices of $\mathbb R$-invariant orthonormal frame $\{\overline L_j\}^n_{j=1}\subset T^{1,0}D$ and hence is globally defined. We have the isometry:
\[\Psi: \Omega^{0,q}(X)\rightarrow\Omega^{n,q}(X,K^*_X).\]
Moreover, it is straightforward to see that 
\begin{equation}\label{e-gue201101yyd}
\ddbar_bu=\Psi^{-1}\ddbar_{b,K^*_X}\Psi u,\ \ \ddbar^*_bu=\Psi^{-1}\ddbar^*_{b,K^*_X}\Psi u,\ \ 
\Box_bu=\Psi^{-1}\Box_{b,K^*_X}\Psi u,\ \ \mbox{for every $u\in\Omega^{0,q}(X)$}. 
\end{equation}
We can now prove 

\begin{thm}\label{t-gue201101yyd}
With the notations used above, we have on $\Omega^{0,\bullet}(X)$, 
	\begin{equation}\label{e-gue201104yyd}
	\begin{split}
\square_b&=\Psi^{-1}\Box_{b,K^*_X}\Psi+
2\cL(\ov L_j,L_k)e_k\wedge i_{L_j}(-\sqrt{-1}T)
+R^{K_X^*}(\ov L_j,L_k)e_k\wedge i_{L_j}\\
&\quad+\Psi^{-1}(\nabla^{K^*_X})^{1,0} 
\mT^*\Psi-\Bigr(\ddbar_b\Psi^{-1}\ov \mT^*\Psi+
\Psi^{-1}\ov \mT^*\Psi\ddbar_b\Bigr),
\end{split}
\end{equation}
where
$\{\overline L_j\}_{j=1}^n$ is a local $\mathbb R$-invariant 
orthonormal frame of $T^{1,0}X$ with dual frame 
$\{\overline e_j \}_{j=1}^n\subset T^{*1,0}X$. 
	\end{thm} 	

\begin{proof}
Let $u\in\Omega^{0,q}(X)$. From \eqref{e-gue201101yyd} and 
\eqref{e-gue201031ycdi}, we have 
\begin{equation}\label{e-gue201105yyd}
\begin{split}
&\Box_bu=\Psi^{-1}\Box_{b,K^*_X}\Psi u\\
&=\Psi^{-1}\overline\Box_{b,K^*_X}\Psi u+
\Psi^{-1}[2\sqrt{-1}\cL,\Lambda](-\sqrt{-1}T)(\Psi u)+
\Psi^{-1}[\sqrt{-1}R^{K^*_X}, \Lambda]\Psi u\\
&\quad+\Psi^{-1}\Bigr((\nabla^{K^*_X})^{1,0}\mT^*+
\mT^*(\nabla^{K^*_X})^{1,0}\Bigr)\Psi u-
\Bigr(\ddbar_b\Psi^{-1}\overline\mT^*\Psi u+
\Psi^{-1}\overline\mT^*\Psi\ddbar_bu\Bigr).
\end{split}
\end{equation}
It is straightforward to check that 
\begin{equation}\label{e-gue201105yydI}
\begin{split}
&[2\sqrt{-1}\cL,\Lambda]=2\cL(\overline L_j,L_k)(\overline e_j\wedge i_{\overline L_k}-i_{L_j}e_k\wedge),\\
&[2\sqrt{-1}R^{K^*_X},\Lambda]=R^{K^*_X}(\overline L_j,L_k)(\overline e_j\wedge i_{\overline L_k}-i_{L_j}e_k\wedge).
\end{split}
\end{equation}
From \eqref{e-gue201105yyd}, \eqref{e-gue201105yydI} and notice that $(\overline e_j\wedge i_{\overline L_k}-i_{L_j}e_k\wedge)v=e_k\wedge i_{L_j}v$, 
$\mT^*(\nabla^{K^*_X})^{1,0}v=0$, for every $v\in\Omega^{n,q}(X,K^*_X)$ and $T$ commutes with $\Psi$, we get \eqref{e-gue201104yyd}. 
\end{proof}

\begin{cor} \label{vanish_0q_1}
With the notations used above, assume that $2\sqrt{-1}\cL=\Theta_X$ 
and there is $C>0$ such that 
\[\sqrt{-1}R^{K_X^*}\geq -C\Theta_X\ \ \mbox{on $X$}.\]
Then, for any $u\in \Omega_c^{0,q}(X)$ with $1\leq q\leq n$, we have
\begin{equation}\label{e-gue201108ycda}
\left(-\sqrt{-1}Tu|u\right)\leq \frac{1}{q}\left(\|\ddbar_b u\|^2+
\|\ddbar_b^* u\|^2\right)+C\|u\|^2.
\end{equation}
\end{cor} 
\begin{proof} 
	Since $2\sqrt{-1}\cL=\Theta_X$, we can choose $\mathbb R$-invariant orthonormal frame $\{\overline L_j\}^n_{j=1}$ 
	such that $2\cL(\ov L_j,L_k)=\delta_{jk}$, for every $j, k$. We write $u=\sum_{J}u_Je_J$ on $D$ with $u_J\in\cC^\infty(D)$ and $e_J=e_{j_1}\wedge\ldots\wedge e_{j_q}$, $j_1<\ldots<j_q$.
	We have
	\begin{equation}\label{e-gue201106ycdh}
\langle\, 2\cL(\ov L_j,L_k)e_k\wedge i_{L_j}(-\sqrt{-1}T)u\,|\,u\,\rangle=\langle\, \sum_J q(-\sqrt{-1}Tu_J)e_J\,|\,u\,\rangle=q\langle\,-\sqrt{-1}Tu\,|\,u\,\rangle.
	\end{equation}
	Since $\sqrt{-1}R^{K_X^*}\geq -C\Theta_X$, as \eqref{e-gue201106ycdh}, we can check that 
	\begin{equation}\label{e-gue201106ycdi}
	\langle R^{K_X^*}(\ov L_j,L_k)e_k\wedge i_{L_j} u|u\rangle\geq -Cq|u|^2.
	\end{equation}
	Since $d\Theta_X=d(2\sqrt{-1}\cL)=0$, we have $\mT=[\Lambda,\dbar_b\Theta_X]=0$. From this observation, \eqref{e-gue201104yyd}, \eqref{e-gue201106ycdh} and 
	\eqref{e-gue201106ycdi}, we obtain 
	\be
	\left(\|\ddbar_b u\|^2+\|\ddbar_b^* u\|^2\right)
	\geq q\left(-\sqrt{-1}Tu|u\right)-qC\|u\|^2
	\ee
	holds for every $u\in \Omega_c^{0,q}(X)$ with $1\leq q\leq n$. 
\end{proof}

		
\section{Szeg\H{o} kernel asymptotics}\label{s-gue201106ycdp}

In this section, we will establish Szeg\H{o} kernel asymptotic 
expansions on $X$ under certain curvature assumptions. 

\subsection{Complete CR manifolds}\label{s-gue201111yyd}
Let $X$ be a CR manifold as in Assumption \ref{as1}.
Let $g_X$ be the $\R$-invariant Hermitian metric as in 
\eqref{e-gue201128yyd}.
We will assume in the following that the Riemannian metric
induced by $g_X$ on $TX$ is complete and study the
extension $\ddbar_b$, $\ddbar_b^*$ and $T$.
We denote by the same symbols the
maximal weak extensions in $L^2$ of these differentials operator.  


Since $g_X$ is complete we know by \cite[Lemma 2.4, p.\,366]{Dem:12} 
that there exists a sequence $\{\chi_k\}^{\infty}_{k=1}\subset\cC^\infty_c(X)$ 
such that $0\leq\chi_k\leq1$, $\chi_{k+1}=1$ on ${\rm supp\,}\chi_k$, 
$|d\chi_k|_g\leq\frac{1}{2^k}$, for every $k=1,2,\ldots$\,, and 
$\bigcup^{\infty}_{k=1}{\rm supp\,}\chi_k=X$. 
By using this sequence as in the Andreotti-Vesentini lemma
on complex Hermitian manifolds 
(cf. \cite[Theorem 3.2, p.\,368]{Dem:12}, \cite[Lemma 3.3.1]{MM})
and the classical Friedrichs lemma we obtain the following. 

\begin{lemma}\label{l-gue201108yydI}
Assume that $(X,g_X)$ is complete. 
Then  $\Omega_c^{p,q}(X)$ is dense in $\Dom(\ddbar_b)$,  
$\Dom(\ddbar_b^*)$, $\Dom T$, $\Dom(\ddbar_b)\cap\Dom(\ddbar_b^*)$ 
and $\Dom(T)\cap\Dom(\ddbar_b)\cap\Dom(\ddbar_b^*)$ with respect to 
the graph norms of $\ddbar_b$, $\ddbar_b^*$, $T$, 
$\ddbar_b+\ddbar_b^*$ and $\ddbar_b+\ddbar_b^*+T$. 
Here the graph-norm of a linear operator $R$ is defined by 
$\|u\|+\|Ru\|$ for $u\in \Dom(R)$. 
\end{lemma}

As a consequence, analogue to \cite[Corollary 3.3.3]{MM}, we obtain:
\begin{cor}
If $(X,g_X)$ be complete, then 
the maximal extension of the formal adjoint of $\ddbar_b$ and $T$ 
coincide with their Hilbert space adjoint, respectively.
\end{cor}

\begin{lemma}\label{l-gue201108yyd}
If $(X,g_X)$ be complete, then 
$\sqrt{-1}T: \Dom(\sqrt{-1}T)\subset 
L^2_{\bullet,\bullet}(X)\rightarrow L^2_{\bullet,\bullet}(X)$ is self-adjoint, that is, 
$(\sqrt{-1}T)^*=\sqrt{-1}T$. 
\end{lemma}

Using these results and  extend the estimates
from Corollary \ref{vanish_0q_1} as follows.

\begin{thm}\label{t-gue201108yyd}
Let $X$ be a CR manifold as in Assumption \ref{as1}.
Assume that $2\sqrt{-1}\cL=\Theta_X$, $g_X$ 
is complete and there is $C>0$ such that 
\[\sqrt{-1}R^{K^*_X}\geq-C\Theta_X.\]
Then, for any $u\in L^2_{0,q}(X)$, $1\leq q\leq n$, 
$u\in\Dom\ddbar_b\bigcap{\rm Dom}\ddbar^*_b\bigcap\Dom(\sqrt{-1}T)$, 
we have 
\begin{equation}\label{e-gue201108ycd}
(-\sqrt{-1}Tu\,|\,u\,)\leq\frac{1}{q}\Bigr(\|\ddbar_bu\|^2+
\|\ddbar^*_bu\|^2\Bigr)+C\|u\|^2.
\end{equation}
\end{thm}

\begin{proof}
Let $u\in L^2_{0,q}(X)$, $1\leq q\leq n$, $u\in\Dom\ddbar_b\bigcap\Dom\ddbar^*_b\bigcap\Dom(\sqrt{-1}T)$. 
From Lemma~\ref{l-gue201108yydI}, we can find  
$\{u_j\}^{\infty}_{j=1}\subset\Omega^{\bullet,\bullet}_c(X)$ such that 
\begin{equation}\label{e-gue201108ycdI}
\lim_{j\infty}\Bigr(\|u_j-u\|^2+\|\ddbar_bu_j-\ddbar_bu\|^2+
\|\sqrt{-1}Tu_j-\sqrt{-1}Tu\|^2\Bigr)=0.
\end{equation}
From \eqref{e-gue201108ycda}, we have for every $j=1,2,\ldots$, 
\begin{equation}\label{e-gue201108ycdII}
(-\sqrt{-1}Tu_j\,|\,u_j\,)\leq\frac{1}{q}\Bigr(\|\ddbar_bu_j\|^2+
\|\ddbar^*_bu_j\|^2\Bigr)+C\|u_j\|^2.
\end{equation}
Taking $j\rightarrow\infty$ in \eqref{e-gue201108ycdII} and 
using \eqref{e-gue201108ycdI}, we get \eqref{e-gue201108ycd}. 
\end{proof}

Let us describe two examples of complete CR manifolds with
complete $\R$-invariant metric $g_X$.

\begin{exam}
Let $(X,HX,J,\omega_0)$ be a compact strictly pseudoconvex
CR manifold as in Assumption \ref{as1} 
and let $g_X$ be an $\R$-invariant metric as in \eqref{e-gue201128yyd}.
Let $\pi:\widetilde{X}\to X$ be a Galois covering of $X$,
that is, there exists a discrete, proper action 
$\Gamma$ such that $X=\widetilde{X}/\Gamma$. 
By pulling back the objects from $X$ by the projection $\pi$
we obtain a strictly pseudoconvex
CR manifold $(\widetilde{X},H\widetilde{X},\widetilde{J},\widetilde{\omega}_0)$ 
satisfying Assumption \ref{as1}.
Moreover, the metric $\widetilde{g}_X=\pi^*g_X$ is a complete
$\R$-invariant metric satisfying \eqref{e-gue201128yyd}.
\end{exam}

\begin{exam}\label{ex:GT}
Let us consider now the case of a circle bundle associated
to a Hermitian holomorphic line bundle.
Let compact complex manifold $(M,J,\Theta_M)$ be a
complete Hermitian manifold.
Let $(L,h^L)\to M$ be a Hermitian holomorphic line bundle over $M$. 
Let $h^{L^*}$ be the Hermitian metric 
on $L^*$ induced by $h^L$. 
Let 
\begin{equation}\label{eq:GraTube}
X:=\{v\in L^*;\, \abs{v}^2_{h^{L^*}}=1\}
\end{equation}
be the circle bundle of $L^*$; 
it is isomorphic to the $S^1$ principal bundle associated to $L$.
Since $X$ is a hypersurface in the complex manifold $L^*$,
it a has a CR structure $(X,HX,J)$ inherited from the complex structure of $L^*$
by setting $T^{1,0}X= TX\cap T^{1,0} L^{*}$.

In this situation,  $S^{1}$ acts on $X$ by fiberwise multiplication,
denoted $(x,e^{i\theta})\mapsto xe^{i\theta}$.
A point $x\in X$ is a pair $x=(p,\lambda)$, where $\lambda$ is a linear
functional on $L_p$, the $S^1$ action is 
$xe^{i\theta}=(p,\lambda)e^{i\theta}=(p,e^{i\theta}\lambda)$.

Let $\omega_0$ be the connection $1$-form on $X$
associated to the Chern connection $\nabla^L$. 
Then
\begin{equation}\label{eq:curvx}
d\omega_0=\pi^*(iR^L),
\end{equation}
where $R^L$ is the curvature of $\nabla^L$, hence $X$ is
a strictly pseudoconvex CR manifold.
Hence $(X,HX,J,\omega_0)$ fulfills Assumption \ref{as1}.
We denote by $\partial_\theta$
the infinitesimal generator of the $S^1$ action on $X$. 
The span of $\partial_\theta$
defines a rank one subbundle $T^VX\cong TS^1\subset TX$, 
the vertical subbundle of the fibration $\pi:X\to M$. 
Moreover \eqref{e-gue201128yyd1} holds for $T=\partial_\theta$.

We construct now a Riemannian metric on $X$. 
Let $g_M$ be a $J$-invariant metric on $TM$ associated to $\Theta_M$.
The Chern connection $\nabla ^L$ on  $L$ induces a connection 
on the $S^1$-principal bundle $\pi:X\to M$, 
and let $T^H X \subset TX$ be the corresponding horizontal bundle.
Let $g_X= \pi^*g_M\oplus d\theta^2/2\pi$
be the metric on 
$TX= T^HX\oplus T S^1$, with $d\theta^2$ 
the standard metric on $S^1= \R/2\pi\Z$.
Then $g_X$ is an $\R$-invariant Hermitian metric on $X$ satisfying
\eqref{e-gue201128yyd}. Since $g_M$ is complete it is
easy to see that $g_X$ is also complete.

\end{exam}

\subsection{The operators $Q_\lambda$, $Q_{[\lambda_1,\lambda]}$, 
$Q_{\tau}$}\label{s-gue201111yydI}

Let $\mathbb{S}$ denote the spectrum of $\sqrt{-1}T$. 
By the spectral theorem, there exists a finite measure $\mu$ on
$\mathbb S\times\mathbb N$ and a unitary operator 
\[U: L^2_{\bullet,\bullet}(X)\rightarrow L^2(\mathbb S\times\mathbb N, d\mu)\]
with the following properties: 
If $h: \mathbb S\times\mathbb N\rightarrow\mathbb R$ 
is the function $h(s,n)=s$, then the element 
$\xi$ of $L^2_{\bullet,\bullet}(X)$ lies in $\Dom(\sqrt{-1}T)$ 
if and only if $hU(\xi)\in L^2(\mathbb S\times\mathbb N, d\mu)$. We have 
\[\mbox{$U\sqrt{-1}TU^{-1}\varphi=h\varphi$, for all $\varphi\in U(\Dom(\sqrt{-1}T))$}.\]
Let $\lambda_1, \lambda\in\mathbb R$, $\lambda_1<\lambda$ 
and let $\tau(t)\in\cC^\infty(\mathbb R,[0,1])$. Put 
\begin{equation}\label{e-gue201108ycde}
\begin{split}
&\cE(\lambda,\sqrt{-1}T):=U^{-1}\Bigr({\rm Image\,}U
\bigcap\{1_{]-\infty,\lambda]}(s)h(s,n);\, h(s,n)\in 
L^2(\mathbb S\times\mathbb N,d\mu)\}\Bigr),\\
&\cE([\lambda_1,\lambda],\sqrt{-1}T):=
U^{-1}\Bigr({\rm Image\,}U\bigcap\{1_{[\lambda_1,\lambda]}(s)h(s,n);\, 
h(s,n)\in L^2(\mathbb S\times\mathbb N,d\mu)\}\Bigr),\\
&\cE(\tau,\sqrt{-1}T):=U^{-1}\Bigr({\rm Image\,}U\bigcap\{\tau(s)h(s,n);\, 
h(s,n)\in L^2(\mathbb S\times\mathbb N,d\mu)\}\Bigr),
\end{split}
\end{equation}
where $1_{]-\infty,\lambda]}(s)=1$ if $s\in]-\infty,\lambda]$,  
$1_{]-\infty,\lambda]}(s)=0$ if $s\notin]-\infty,\lambda]$ 
and similar for $1_{[\lambda_1,\lambda]}(s)$. 
Let 
\begin{equation}\label{e-gue201108ycdf}
\begin{split}
&Q_\lambda: L^2_{\bullet,\bullet}(X)\rightarrow\cE(\lambda,\sqrt{-1}T),\\
&Q_{[\lambda_1,\lambda]}: L^2_{\bullet,\bullet}(X)\rightarrow\cE([\lambda_1,\lambda],\sqrt{-1}T),\\
&Q_{\tau}: L^2_{\bullet,\bullet}(X)\rightarrow\cE(\tau_\lambda,\sqrt{-1}T)
\end{split}
\end{equation}
be the orthogonal projections with respect to $(\,\cdot\,|\,\cdot\,)$.  

Since $X$ is strictly pseudoconvex, from~\cite[Lemma 3.4 (3), p.\,239]{KN96}, 
\cite[Theorem 3.5]{HHL17}, we have one of the following two cases: 
\begin{equation}\label{e-gue201109yyd}
\begin{split}
&\mbox{(a) The $\mathbb R$-action is free},\\
&\mbox{(b) The $\mathbb R$-action comes from a CR torus action $\mathbb T^d$ on $X$ and $\omega_0$ and $\Theta_X$ are $\mathbb T^d$ invariant}. 
\end{split}
\end{equation}

Assume that the $\mathbb R$-action is free. Let $D=U\times I$ be a BRT chart with BRT coordinates $x=(x_1,\ldots,x_{2n+1})$. Since the $\mathbb R$-action is 
free, we can extend $x=(x_1,\ldots,x_{2n+1})$ to $\hat D:=U\times\mathbb R$. We identify $\hat D$ with an open set in $X$.

\begin{lemma}\label{l-gue201110yyd}
Assume that the $\mathbb R$-action is free. Let $D=U\times I$ be a BRT chart with BRT coordinates $x=(x_1,\ldots,x_{2n+1})$. Let $\lambda_1, \lambda\in\mathbb R$, $\lambda_1<\lambda$. For $u\in\Omega^{\bullet,\bullet}_c(D)$, we have 
\begin{equation}\label{e-gue201110yyd}
(Q_\lambda u)(x)=\frac{1}{(2\pi)^{2n+1}}\int e^{i<x-y,\eta>}1_{]-\infty,\lambda]}(-\eta_{2n+1})u(y)dyd\eta\in\Omega^{\bullet,\bullet}(\hat D),
\end{equation}
\begin{equation}\label{e-gue201110yydI}
(Q_{[\lambda_1,\lambda]}u)(x)=\frac{1}{(2\pi)^{2n+1}}\int e^{i<x-y,\eta>}1_{[\lambda_1,\lambda]}(-\eta_{2n+1})u(y)dyd\eta\in\Omega^{\bullet,\bullet}(\hat D),
\end{equation}
\begin{equation}\label{e-gue201110yydII}
(Q_{\tau}u)(x)=\frac{1}{(2\pi)^{2n+1}}\int e^{i<x-y,\eta>}\tau(-\eta_{2n+1})u(y)dyd\eta\in\Omega^{\bullet,\bullet}(\hat D),
\end{equation}
and ${\rm supp\,}Q_\lambda u\subset\hat D$, ${\rm supp\,}Q_{[\lambda_1,\lambda]}u\subset\hat D$, ${\rm supp\,}Q_{\tau_\lambda}u\subset\hat D$, where $\hat D$ is in the discussion after \eqref{e-gue201109yyd}. 
\end{lemma}

\begin{proof}
Let $\chi\in\cC^\infty_c(\mathbb R)$, $\chi=1$ on $[-1,1]$, $\chi=0$ outside $[-2,2]$. For every $M>0$, put $\tau_{M}(t):=\chi(\frac{t}{M})\tau(t)$. Then,
\begin{equation}\label{e-gue201110ycd}
\mbox{$Q_{\tau}u=\lim_{M\rightarrow\infty}Q_{\tau_{M}}u$ in 
$L^2_{\bullet,\bullet}(X)$, for every $u\in L^2_{\bullet,\bullet}(X)$}.
\end{equation}
From the Helffer-Sj\"ostrand formula~\cite[Proposition 7.2]{HS89}, we see that 
\begin{equation}\label{e-gue201110yydIII}
Q_{\tau_{M}}=\frac{1}{2\pi i}
\int_{\mathbb C}\frac{\partial\tilde\tau_{M}}{\partial\overline z}(z-\sqrt{-1}T)^{-1}
dz\wedge d\overline z\ \ \mbox{on $L^2_{\bullet,\bullet}(X)$},
\end{equation}
where $\tilde\tau_{M}\in\cC^\infty_c(\mathbb C)$ is an extension of 
$\tau_{M}$ with $\frac{\partial\tilde\tau_{M}}{\partial\overline z}=0$ 
on $\mathbb R$. It is not difficult to see that for $u\in\Omega^{\bullet,\bullet}_c(D)$, 
\begin{equation}\label{e-gue201110yydIV}
(z-\sqrt{-1}T)^{-1}u=\frac{1}{(2\pi)^{2n+1}}
\int e^{i<x-y,\eta>}\frac{1}{z+\eta_{2n+1}}u(y)dyd\eta
\in\Omega^{\bullet,\bullet}(\hat D)
\end{equation}
and ${\rm supp\,}(z-\sqrt{-1}T)^{-1}u\subset\hat D$. 
From \eqref{e-gue201110yydIII} and \eqref{e-gue201110yydIV}, we have 
\begin{equation}\label{e-gue201110yydV}
(Q_{\tau_{M}}u)(x)=\frac{1}{2\pi i}\frac{1}{(2\pi)^{2n+1}}
\int_{\mathbb C}\int e^{i<x-y,\eta>}\frac{\frac{\partial\tilde\tau_{M}}
{\partial\overline z}}{z+\eta_{2n+1}}u(y)dyd\eta dz\wedge d\overline z
\in\Omega^{\bullet,\bullet}(\hat D) 
\end{equation}
and ${\rm supp\,}Q_{\tau_{M}}u\subset\hat D$, for every 
$u\in\Omega^{\bullet,\bullet}_c(D)$. 
By the Cauchy integral formula, we see that 
\[\frac{1}{2\pi i}\int_{\mathbb  C}\frac{1}{z+\eta_{2n+1}}
\frac{\partial\tilde\tau_{M}}{\partial\overline z}
dz\wedge d\overline z=\tau_{M}(-\eta_{2n+1}).\]
From this observation and \eqref{e-gue201110yydV}, we deduce that 
\begin{equation}\label{e-gue201110ycdI}
(Q_{\tau_{M}}u)(x)=\frac{1}{(2\pi)^{2n+1}}
\int e^{i<x-y,\eta>}\tau_{M}(-\eta_{2n+1})u(y)dyd\eta
\in\Omega^{\bullet,\bullet}(\hat D) 
\end{equation}
and ${\rm supp\,}Q_{\tau_{M}}u\subset\hat D$, for every 
$u\in\Omega^{\bullet,\bullet}_c(D)$. 
From \eqref{e-gue201110ycd} and \eqref{e-gue201110ycdI}, 
we get \eqref{e-gue201110yydII}. 

Let $\gamma_\varepsilon\in\cC^\infty_c(\mathbb R)$, $\lim_{\varepsilon\rightarrow0}\gamma_\varepsilon(t)=
1_{]-\infty,\lambda]}(t)$, for every $t\in\mathbb R$. 
We can repeat the proof above and get that 
\[Q_\lambda u=\lim_{\varepsilon\rightarrow0}Q_{\gamma_\varepsilon}u=
\frac{1}{(2\pi)^{2n+1}}
\int e^{i<x-y,\eta>}1_{]-\infty,\lambda]}(-\eta_{2n+1})u(y)dyd\eta
\in\Omega^{\bullet,\bullet}(\hat D)\]
and ${\rm supp\,}Q_{\lambda}u\subset\hat D$, 
for every $u\in\Omega^{\bullet,\bullet}_c(D)$.  
We obtain \eqref{e-gue201110yyd}. The proof of \eqref{e-gue201110yydI} is similar. 
\end{proof}

We now assume that the $\mathbb R$-action is not free. 
From \eqref{e-gue201109yyd}, we know that the $\mathbb R$-action 
comes from a CR torus action 
$\mathbb T^d=(e^{i\theta_1},\ldots,e^{i\theta_d})$ on $X$ and $\omega_0$, 
$\Theta_X$ are $\mathbb T^d$ invariant. Since the $\mathbb R$-action comes 
from the $\mathbb T^d$-action, 
there exist $\beta_1,\ldots,\beta_d\in\mathbb R$, such that 
\begin{equation}\label{e-gue201116yyda}
T=\beta_1T_1+\ldots+\beta_dT_d, 
\end{equation}
where $T_j$ is the vector field on $X$ given by $T_ju:=
\frac{\partial}{\partial\theta_j}((1,\ldots,1,e^{i\theta_j},1,\ldots,1)^*u)|_{\theta_j=0}$, 
$u\in\Omega^{\bullet,\bullet}(X)$, $j=1,\ldots,d$. For $(m_1,\ldots,m_d)\in\mathbb Z^d$, 
put 
\[
L^{2,m_1,\ldots,m_d}_{\bullet,\bullet}(X)\\
:=\{u\in L^2_{\bullet,\bullet}(X);\, (e^{i\theta_1},\ldots,e^{i\theta_d})^*u=
e^{im_1\theta_1+\ldots+im_d\theta_d}u, \forall\, 
(\theta_1,\ldots,\theta_d)\in\R^d\}
\]
and let
\begin{equation}\label{e-gue201112yyd}
Q_{m_1,\ldots,m_d}: L^2_{\bullet,\bullet}(X)\rightarrow 
L^{2,m_1,\ldots,m_d}_{\bullet,\bullet}(X)
\end{equation}
be the orthogonal projection. It is not difficult to see that for every $u\in L^2_{\bullet,\bullet}(X)$, we have 
\begin{equation}\label{e-gue201112yydI}
\begin{split}
&Q_\lambda u=\sum_{\substack{(m_1,\ldots,m_d)\in\mathbb Z^d,\\ 
-m_1\beta_1-\ldots-m_d\beta_d\leq\lambda}}Q_{m_1,\ldots,m_d}u,\\
&Q_{[\lambda_1,\lambda]}u=\sum_{\substack{(m_1,\ldots,m_d)\in\mathbb Z^d,\\ 
\lambda_1\leq-m_1\beta_1-\ldots-m_d\beta_d\leq\lambda}}Q_{m_1,\ldots,m_d}u,\\
&Q_\tau u=\sum_{(m_1,\ldots,m_d)\in\mathbb Z^d}
\tau(-m_1\beta_1-\ldots-m_d\beta_d)Q_{m_1,\ldots,m_d}u.
\end{split}
\end{equation}
From Lemma~\ref{l-gue201110yyd} and \eqref{e-gue201112yydI}, 
we conclude that 

\begin{prop}\label{p-gue201112yyd}
Let $\lambda_1, \lambda\in\mathbb R$, $\lambda_1<\lambda$. 
For $u\in\Dom\ddbar_b$, we have $Q_\lambda u, 
Q_{[\lambda_1,\lambda]}u, Q_\tau u\in\Dom\ddbar_b$ 
and $\ddbar_bQ_\lambda u=Q_\lambda\ddbar_bu$, 
$\ddbar_bQ_{[\lambda_1,\lambda]} u=Q_{[\lambda_1,\lambda]}\ddbar_bu$, 
$\ddbar_bQ_\lambda u=Q_\lambda \ddbar_bu$. 
 Similarly, for $u\in\Dom\ddbar^*_b$, we have 
$Q_\lambda u, Q_{[\lambda_1,\lambda]}u, 
Q_\tau u\in\Dom\ddbar^*_b$ and $\ddbar^*_bQ_\lambda u=
Q_\lambda\ddbar^*_bu$, $\ddbar^*_bQ_{[\lambda_1,\lambda]} u=
Q_{[\lambda_1,\lambda]}\ddbar^*_bu$, $\ddbar^*_bQ_\tau u=
Q_\tau\ddbar^*_bu$. 
\end{prop}

For $\lambda\in\mathbb R$, define 
\begin{equation}\label{e-gue201112yydII}
\begin{split}
&\Box_{b,\lambda}: \Dom\Box_{b,\lambda}\subset\cE(\lambda,\sqrt{-1}T)\rightarrow\cE(\lambda,\sqrt{-1}T),\\
\Dom\Box_{b,\lambda}&:=\Dom\Box_{b}\bigcap \cE(\lambda,\sqrt{-1}T),\:\:
\Box_{b,\lambda}u=\Box_bu,\ \ \text{for }u\in\Dom\Box_{b,\lambda},
\end{split}
\end{equation}
where $\Box_{b}$ is defined in \eqref{e:gaf1}, \eqref{e:gaf2}.
From Proposition~\ref{p-gue201112yyd}, we see that 
\begin{equation}\label{e-gue201112yydIII}
\begin{split}
&\Dom\Box_{b,\lambda}=Q_\lambda(\Dom\Box_b),\\
&Q_\lambda\Box_b=\Box_bQ_\lambda=\Box_{b,\lambda}Q_\lambda\ \ 
\mbox{on $\Dom\Box_b$}. 
\end{split}
\end{equation}
From now on, we write $\Box^{(q)}_b$ and $\Box^{(q)}_{b,\lambda}$ 
to denote $\Box_b$ and $\Box_{b,\lambda}$ acting on $(0,q)$ forms, 
respectively. 

\subsection{Local closed range for $\Box^{(0)}_b$}\label{s-gue201112yyd}

In this section, we will establish the local closed range property for 
$\Box^{(0)}_b$ under appropriate curvature assumptions. 
We first need the following.

\begin{lemma}\label{l-gue201112yydh}
Assume that $2\sqrt{-1}\cL=\Theta_X$, $g_X$ is complete and there is $C>0$ such that 
\[\sqrt{-1}R^{K^*_X}\geq-C\Theta_X.\]
Then, for any $u\in L^2_{0,q}(X)$, $1\leq q\leq n$, 
$u\in\Dom\ddbar_b\bigcap\Dom\ddbar^*_b\bigcap\cE(\lambda,\sqrt{-1}T)$, 
$\lambda\leq -2C$, we have 
\begin{equation}\label{e-gue201112yydh}
\|u\|^2\leq\frac{1}{qC}\Bigr(\|\ddbar_bu\|^2+\|\ddbar^*_bu\|^2\Bigr). 
\end{equation}
\end{lemma}

\begin{proof}
Let $\lambda\leq-2C$ and let $u\in\Dom\ddbar_b\bigcap
\Dom\ddbar^*_b\bigcap\cE(\lambda,\sqrt{-1}T)$, 
$u\in L^2_{0,q}(X)$, $1\leq q\leq n$. 
Let $M\gg1$ and let $u_M:=Q_{[-M,\lambda]}u$. 
By Proposition~\ref{p-gue201112yyd}, we see that 
\[u_M\in\Dom\ddbar_b\bigcap\Dom\ddbar^*_b\bigcap
\cE(\lambda,\sqrt{-1}T)\bigcap\Dom(\sqrt{-1}T).\]
From this observation and \eqref{e-gue201108ycd}, we have 
\begin{equation}\label{e-gue201113yyd}
\begin{split}
&-\lambda\|u_M\|^2\leq(\,-\sqrt{-1}Tu_M\,|\,u_M\,)\leq\frac{1}{q}\Bigr(\|Q_{[-M,\lambda]}\ddbar_bu\|^2+\|Q_{[-M,\lambda]}\ddbar^*_bu\|^2\Bigr)+C\|u_M\|^2\\
&\leq\frac{1}{q}\Bigr(\|\ddbar_bu\|^2+\|\ddbar^*_bu\|^2\Bigr)+C\|u_M\|^2.
\end{split}
\end{equation}
Note that $\lambda\leq-2C$. From this observation and \eqref{e-gue201113yyd}, we deduce that 
\begin{equation}\label{e-gue201113yydI}
\|u_M\|^2\leq\frac{1}{qC}\Bigr(\|\ddbar_bu\|^2+\|\ddbar^*_bu\|^2\Bigr).
\end{equation}
Let $M\rightarrow\infty$ in  \eqref{e-gue201113yydI}, we get \eqref{e-gue201112yydh}. 
\end{proof}

For every $q=0,1,\ldots,n$, put $\cE^{(q)}(\lambda,\sqrt{-1}T):=
\cE(\lambda,\sqrt{-1}T)\bigcap L^2_{0,q}(X)$. We can now prove: 

\begin{thm}\label{t-gue201113yyd}
Assume that $2\sqrt{-1}\cL=\Theta_X$, $g_X$ is complete and 
there is $C>0$ such that 
\[\sqrt{-1}R^{K^*_X}\geq-C\Theta_X.\]
Let $q\in\{1,\ldots,n\}$. Let $\lambda\in\mathbb R$, $\lambda\leq-2C$. 
The operator 
\[\Box^{(q)}_{b,\lambda}: \Dom\Box^{(q)}_{b,\lambda}\subset\cE^{(q)}
(\lambda,\sqrt{-1}T)\rightarrow\cE^{(q)}(\lambda,\sqrt{-1}T)\]
has closed range and ${\rm Ker\,}\Box^{(q)}_{b,\lambda}=\{0\}$. 
Hence, there is a bounded operator 
\[G^{(q)}_\lambda: \cE^{(q)}(\lambda,\sqrt{-1}T)\rightarrow
\Dom\Box^{(q)}_{b,\lambda}\] 
such that 
\begin{equation}\label{e-gue201113ycd}
\Box^{(q)}_{b,\lambda}G^{(q)}_\lambda=I\ \ \mbox{on $\cE^{(q)}(\lambda,\sqrt{-1}T)$}. 
\end{equation}
\end{thm}

\begin{proof}
Let $u\in\Dom\Box^{(q)}_{b,\lambda}$. 
From \eqref{e-gue201112yydh}, we have 
\[\|u\|^2\leq\frac{1}{qC}\Bigr(\|\ddbar_bu\|^2+\|\ddbar^*_bu\|^2\Bigr)=\frac{1}{qC}(\,\Box^{(q)}_{b,\lambda}u\,|\,u\,).\]
Hence, 
\begin{equation}\label{e-gue201113yydII}
\|u\|\leq\frac{1}{qC}\|\Box^{(q)}_{b,\lambda}u\|.
\end{equation}
From \eqref{e-gue201113yydII}, the theorem follows. 
\end{proof}

We now consider $(0,1)$ forms. Let $G^{(1)}_\lambda$ be as in \eqref{e-gue201113ycd}. Since $G^{(1)}_\lambda$ is $L^2$ bounded, 
there is $C_0>0$ such that 
\begin{equation}\label{e-gue201113ycdI}
\|G^{(1)}_\lambda v\|\leq C_0\|v\|,\ \ \mbox{for every $v\in\cE^{(1)}(\lambda,\sqrt{-1}T)$}. 
\end{equation}
We can now prove 

\begin{thm}\label{t-gue201113yydI}
Assume that $2\sqrt{-1}\cL=\Theta_X$, $g_X$ is complete and 
there is $C>0$ such that 
\[\sqrt{-1}R^{K^*_X}\geq-C\Theta_X.\]
Let $\lambda\in\mathbb R$, $\lambda\leq-2C$. 
For every $v\in\cE^{(1)}(\lambda,\sqrt{-1}T)$ with $\ddbar_bv=0$, 
we can find $u\in\Dom\ddbar_b\bigcap\cE^{(0)}(\lambda,\sqrt{-1}T)$ such that 
\begin{equation}\label{e-gue201113ycdII}
\begin{split}
&\ddbar_bu=v,\\
&\|u\|^2\leq C_0\|v\|^2, 
\end{split}
\end{equation}
where $C_0>0$ is a constant as in \eqref{e-gue201113ycdI}. 
\end{thm}

\begin{proof}
Let $v\in\cE^{(1)}(\lambda,\sqrt{-1}T)$ with $\ddbar_bv=0$. From \eqref{e-gue201113ycd}, we have 
\begin{equation}\label{e-gue201113ycdIII}
v=\ddbar_b\,\ddbar^*_bG^{(1)}_\lambda v+\ddbar^*_b\ddbar_bG^{(1)}_\lambda v. 
\end{equation}
Since $\ddbar_b\Bigr(\ddbar^*_b\ddbar_bG^{(1)}_\lambda v\Bigr)=
\ddbar_bv-\ddbar^2_b\,\ddbar^*_bG^{(1)}_\lambda v=0$, 
$\ddbar^*_b\Bigr(\ddbar^*_b\ddbar_bG^{(1)}_\lambda v\Bigr)=0$, 
$\ddbar^*_b\ddbar_bG^{(1)}_\lambda v\in{\rm Ker\,}\Box^{(1)}_{b,\lambda}$. 
From Theorem~\ref{t-gue201113yyd}, we see that 
$\ddbar^*_b\ddbar_bG^{(1)}_\lambda v=0$. From this observation 
and \eqref{e-gue201113ycdIII}, 
we get $v=\ddbar_bu$, $u=\ddbar^*_bG^{(1)}_\lambda v$. Now, 
\[\|u\|^2=\|\ddbar^*_bG^{(1)}_\lambda v\|^2\leq
\|\ddbar_bG^{(1)}_\lambda v\|^2+\|\ddbar^*_bG^{(1)}_\lambda v\|^2=
(\,\Box^{(1)}_{b,\lambda}G^{(1)}_\lambda v\,|\,G^{(1)}_\lambda v\,)=
(\,v\,|\,G^{(1)}_\lambda v\,)\leq C_0\|v\|^2,\]
where $C_0>0$ is as in \eqref{e-gue201113ycdI}. The theorem follows. 
\end{proof}

Fix $q\in\{0,1,\ldots,n\}$. Let 
\[S^{(q)}: L^2_{0,q}(X)\rightarrow \Ker\square_b^{(q)}\]
be the orthogonal projection with respect to $(\,\cdot\,|\,\cdot\,)$. From Proposition~\ref{p-gue201112yyd}, we can check that
\begin{equation}\label{e-gue201115yyd}
\begin{split}
&\mbox{$Q_\lambda S^{(q)}=S^{(q)}Q_\lambda$ on $L^2_{0,q}(X)$},\\
&\mbox{$Q_{[\lambda_1,\lambda]}S^{(q)}=S^{(q)}Q_{[\lambda_1,\lambda]}$ on $L^2_{0,q}(X)$},\\
&\mbox{$Q_\tau S^{(q)}=S^{(q)}Q_\tau$ on $L^2_{0,q}(X)$}.
\end{split}
\end{equation}
We recall the following notion introduced in~\cite[Definition 1.8]{HM16}. 

\begin{defn}\label{d-gue201114yydf}
Fix $q\in\{0,1,2,\ldots,n\}$. Let $Q:L^2_{0,q}(X)\rightarrow L^2_{0,q}(X)$
be a continuous operator. We say that $\Box^{(q)}_b$ has local $L^2$ closed range
on an open set $D\subset X$ with respect to $Q$ if for every $D'\Subset D$, 
there exist constants $C_{D'}>0$ and $p\in\mathbb N$, such that
\[\|Q(I-S^{(q)})u\|^2\leq C_{D'}\big(\,(\Box^{(q)}_b)^pu\,|\,u\big),\ \ \mbox{for all $u\in\Omega^{0,q}_c(D')$}.\]
\end{defn}

We remind the reader that we do not assume that $\Theta_X=2\sqrt{-1}\cL$.
$2\sqrt{-1}\cL$ induces a Hermitian metric $\langle\,\cdot\,|\,\cdot\,\rangle_{\cL}$ on $\mathbb CTX$ and $\langle\,\cdot\,|\,\cdot\,\rangle_{\cL}$ induces a Hermitian metric $\langle\,\cdot\,|\,\cdot\,\rangle_{\cL}$ on $T^{*\bullet,\bullet}X$.
More precisely, if $X$ is strictly pseudoconvex, i.e., $2\sqrt{-1}\cL\in\Omega^{1,1}(X)$ is positive-definite, then we can construct a Hermitian metric $\langle\,\cdot\,|\,\cdot\,\rangle_{\cL}$ on $\mathbb CTX=T^{1,0}X\oplus T^{0,1}X\oplus \C \{T\}$ in the following way:
For arbitrary $a,b\in T^{1,0}X$, $\langle a| b\rangle_\cL:=2\cL(a,\ov b)$,  $\langle \ov a| \ov b\rangle_\cL:=\langle b| a\rangle_\cL$, $\langle a| \ov b\rangle_\cL:=0$ and $\langle T| T\rangle_\cL:=1$.
We simply use 
$2\sqrt{-1}\cL$ to represent $\langle\,\cdot\,|\,\cdot\,\rangle_{\cL}$. Let $(\,\cdot\,|\,\cdot\,)_{\cL}$ be the $L^2$ inner product on $\Omega^{\bullet,\bullet}_c(X)$ induced by $\langle\,\cdot\,|\,\cdot\,\rangle_{\cL}$ and let $L^2_{\bullet,\bullet}(X,\cL)$ be the completion of $\Omega^{\bullet,\bullet}_c(X)$ with respect to $(\,\cdot\,|\,\cdot\,)_{\cL}$. We write $L^2(X,\cL):=L^2_{0,0}(X,\cL)$. For $f\in  L^2_{\bullet,\bullet}(X,\cL)$, we write $\|f\|^2_{\cL}:=(\,f\,|\,f\,)_{\cL}$. 
  
Let $R^{K_X^*}_{\cL}$ be the Chern curvature of $K^*_X$ with respect to the Hermitian metric $\langle,\rangle_{\cL}$ on $X$, see (\ref{e-gue201025yyd}). Locally it can be represented by
\begin{equation}\label{e-201120curc}
	R^{K_X^*}_{\cL}=\ddbar_b\dbar_b\log\det\left(\langle Z_j|Z_k \rangle_{\cL}\right)_{j,k=1}^n.
\end{equation}
  
For $u\in\Omega^{\bullet,\bullet}_c(X)$, from Lemma~\ref{l-gue201110yyd} and \eqref{e-gue201112yydI}, we see that $Q_\lambda u$, $Q_{[\lambda_1,\lambda]}u$, $Q_\tau u$ are independent of the choices of $\mathbb R$-invariant Hermitian metrics on $X$. We can now prove 
  
\begin{thm}\label{t-gue201113yydp} 
Assume that $2\sqrt{-1}\cL$ is complete and there is $C>0$ such that 
{\begin{equation}\label{e-gue201113yydm}
\sqrt{-1}R^{K^*_X}_{\cL}\geq-2C\sqrt{-1}\cL,\ \ 
{(2\sqrt{-1}\cL)^n\wedge\omega_0\geq C\Theta_X^n\wedge\omega_0.}
\end{equation}}
Let $D\Subset X$ be an open set. Let $\lambda\in\mathbb R$, 
$\lambda\leq-2C$. Then, $\square_b^{(0)}$ has local closed range on $D$ 
with respect to $Q_\lambda$.
\end{thm}

\begin{proof}
Let $u\in\cC^\infty_c(D)$. Let $v:=\ddbar_bQ_\lambda u=Q_\lambda\ddbar_bu$. Since $\ddbar_bu\in\Omega^{0,1}_c(D)$, 
\[Q_\lambda\ddbar_bu\in L^2_{\bullet,\bullet}(X,\cL)\bigcap L^2_{\bullet,\bullet}(X).\] 
From Theorem~\ref{t-gue201113yydI}, there exists $g\in L^2(X,\cL)$ with 
\begin{equation}\label{e-gue201113yydk}
\|g\|^2_{\cL}\leq C_0\|\ddbar_bQ_\lambda u\|^2_{\cL}\leq C_0\|\ddbar_bu\|^2_{\cL}
\end{equation}
such that 
\begin{equation}\label{e-gue201113yydl}
\ddbar_bg=\ddbar_bQ_\lambda u,
\end{equation}
where $C_0>0$ is a constant as in \eqref{e-gue201113ycdII}. 
Since $\ddbar_b(I-S^{(0)})Q_\lambda u=\ddbar_bQ_\lambda u$ 
and $(I-S^{(0)})Q_\lambda u\perp{\rm Ker\,}\ddbar_b$, we have 
\begin{equation}\label{e-gue201113yydn}
\|(I-S^{(0)})Q_\lambda u\|^2\leq\|g\|^2\leq C\|g\|^2_{\cL}, 
\end{equation}
where $C>0$ is a constant as in \eqref{e-gue201113yydm}. 
From \eqref{e-gue201113yydm}, \eqref{e-gue201113yydk} and 
\eqref{e-gue201113yydn}, we have 
\begin{equation}\label{e-gue201113yydp}
\|Q_\lambda(I-S^{(0)})u\|^2=
\|(I-S^{(0)})Q_\lambda u\|^2\leq C\|g\|^2_{\cL}\leq CC_0\|\ddbar_bu\|^2_{\cL}. 
\end{equation}
Since $\ddbar_bu$ has compact suppprt in $D$, there exists $C_1>0$ independent of $u$ such that 
\begin{equation}\label{e-gue201113yydq}
\|\ddbar_bu\|^2_{\cL}\leq C_1\|\ddbar_bu\|^2.
\end{equation}
From \eqref{e-gue201113yydp} and 
\eqref{e-gue201113yydq}, the theorem follows. 
\end{proof}

For $\lambda\in\mathbb R$, $\lambda\leq0$, let 
$\tau_\lambda\in\cC^\infty(\mathbb R,[0,1])$, 
$\tau_\lambda=1$ on $]-\infty, 2\lambda]$, $\tau_\lambda=0$ outside 
$]-\infty,\lambda]$. It is clear that 
$\|Q_{\tau_\lambda}(I-S^{(0)})u\|\leq\|Q_{\lambda}(I-S^{(0)})u\|$, 
for every $u\in L^2(X)$. From this observation and 
Theorem~\ref{t-gue201113yydp}, we deduce that 

\begin{thm}\label{t-gue201115yyd}
Assume that $2\sqrt{-1}\cL$ is complete and there is $C>0$ such that 
{\begin{equation}\label{e-gue201115yydI}
\sqrt{-1}R^{K^*_X}_{\cL}\geq-2C\sqrt{-1}\cL,\ \ 
{(2\sqrt{-1}\cL)^n\wedge\omega_0\geq C\Theta_X^n\wedge\omega_0.}
\end{equation}}
Let $D\Subset X$ be an open set. Let $\lambda\in\mathbb R$, 
$\lambda\leq-2C$. Then, $\square_b^{(0)}$ has local closed range on 
$D$ with respect to $Q_{\tau_\lambda}$.
\end{thm}
  
\subsection{Vanishing theorems}
In this section we present some vanishing theorems
that follow from the previous $L^2$ estimates.
We obtain first a CR counterpart of the Kodaira vanishing theorem 
\cite[Theorem 1.5.4.(a)]{MM} as follows.

\begin{cor}
Assume that $2\sqrt{-1}\cL=\Theta_X$, $g_X$ is complete and 
let $\lambda<0$ and $1\leq q\leq n$. 
Then, we have
\be
\Ker\square_b \cap\cE(\lambda,\sqrt{-1}T)\cap L^2_{n,q}(X)=0.
\ee 
\end{cor}   
\begin{proof}
From Corollary \ref{vanish_nq_1}, we have for any $u\in \Omega_c^{n,q}(X)$,
\be
\left(-\sqrt{-1}Tu|u\right)\leq 
\frac{1}{q}\left(\|\ddbar_b u\|^2+\|\ddbar^*_b u\|^2\right).
\ee
Since $X$ is complete, $\sqrt{-1}T$ is self-adjoint, and for 
any $u\in\Dom \sqrt{-1}T\cap\Dom\ddbar_b\cap\Dom\ddbar_b^*$ 
there exists a sequence $\{u_m\}_{m\geq 1}\subset \Omega_c^{\bullet,\bullet}(X)$ 
converges to $u$ with respect to the graph norm of $\ddbar_b+\ddbar_b^*+\sqrt{-1}T$. 
Thus, for any $u\in\Dom(\ddbar_b)\cap\Dom(\ddbar^*_b)
\cap\Dom(\sqrt{-1}T)\cap\cE(\lambda,\sqrt{-1}T)\cap L^2_{n,q}(X)$, 
we obtain
\be
-\lambda\|u\|^2\leq\left(-\sqrt{-1}Tu|u\right)\leq 
\frac{1}{q}\left(\|\ddbar_b u\|^2+\|\ddbar^*_b u\|^2\right).
\ee
Assume $-\lambda>0$. Let $v\in \Dom(\ddbar_b)\cap\Dom(\ddbar_b)\cap
\cE(\lambda,\sqrt{-1}T)\cap L^2_{n,q}(X)$. By the completeness of $X$, 
there exists a sequence $\{ v_m \}_{m\geq 1}\in\Omega_c^{n,q}(X)$ such 
that $\|\ddbar_bv-\ddbar_bv_m\|+
\|\ddbar^*_bv-\ddbar^*_bv_m\|+\|v-v_m\|\rightarrow 0$ as 
$n\rightarrow \infty$. It is clear that $Q_\lambda v_m\in\cE(\lambda,\sqrt{-1}T)$, 
and
\be
\|\ddbar_bQ_\lambda v_m\|=\|Q_\lambda\ddbar_b v_m\|\leq 
\|\ddbar_b v_m\|<\infty,\\
\|\ddbar^*_bQ_\lambda v_m\|=
\|Q_\lambda\ddbar^*_b v_m\|\leq \|\ddbar^*_b v_m\|<\infty,\\
\|\sqrt{-1}TQ_\lambda v_m\|=
\|Q_\lambda\sqrt{-1}Tv_m\|\leq \|\sqrt{-1}Tv_m\|< \infty.
\ee
Here the last inequality follows from that 
$Q_\lambda(\sqrt{-1}T)\subset (\sqrt{-1}T)Q_\lambda$, i.e., 
$Q_\lambda\Dom(\sqrt{-1}T)\subset \Dom(\sqrt{-1}T)$ and 
$Q_\lambda(\sqrt{-1}T)= (\sqrt{-1}T)Q_\lambda$ on $\Dom(\sqrt{-1}T)$. 
Thus, we conclude that $Q_\lambda v_m\in\Dom(\ddbar_b)\cap\Dom(\ddbar^*_b)\cap\Dom(\sqrt{-1}T)\cap\cE(\lambda,\sqrt{-1}T)\cap 
L^2_{n,q}(X)$ for $m=1,2,\cdots$. Moreover, 
\be  
\begin{split}
-\lambda\|Q_\lambda v_m\|^2
&\leq \frac{1}{q}\left(\|\ddbar_b Q_\lambda v_m\|^2+
\|\ddbar^*_b Q_\lambda v_m\|^2\right)\\
&=\frac{1}{q}\left(\| Q_\lambda\ddbar_b v_m\|^2+
\| Q_\lambda\ddbar_b^* v_m\|^2\right)\\
&\leq \frac{1}{q}\left(\|\ddbar_b v_m\|^2+\|\ddbar^*_b v_m\|^2\right)
\end{split}
\ee
Note that $\|v-v_m\|^2=\|v-Q_\lambda v_m\|^2+
\|(Q_\lambda-1) v_m\|^2\rightarrow 0$ as $m\rightarrow \infty$. 
We obtain 
\be
-\lambda\|v\|^2\leq \frac{1}{q}\left(\|\ddbar_b v\|^2+\|\ddbar^*_b v\|^2\right)
\ee
for all $v\in \Dom(\ddbar_b)\cap\Dom(\ddbar_b^*)\cap
\cE(\lambda,\sqrt{-1}T)\cap L^2_{n,q}(X)$. Finally, we use $-\lambda>0$.
\end{proof}  
We obtain a CR counterpart of the Kodaira-Serre vanishing theorem 
\cite[Theorem 1.5.6]{MM} as follows.
\begin{cor}
Assume that $2\sqrt{-1}\cL=\Theta_X$, $g_X$ is complete and let $C\geq 0$ such that	
\be
2\sqrt{-1}\cL=\Theta_X, \quad \sqrt{-1}R^{K_X^*}\geq -C\Theta_X.
\ee 
Let $\lambda<-C$ and $1\leq q\leq n$. Then, we have
	\be
	\Ker\square_b\cap \cE(\lambda,\sqrt{-1}T)\cap L^2_{0,q}(X)=0.   
	\ee  
\end{cor}     

\begin{proof}
	By Corollary \ref{vanish_0q_1}, we have
	for any $u\in \Omega_c^{0,q}(X)$ with $1\leq q\leq n$, we have
	\be
	\left(-\sqrt{-1}Tu|u\right)\leq \frac{1}{q}\left(\|\ddbar_b u\|^2+\|\ddbar_b^* u\|^2\right)
	+C\|u\|^2.
	\ee
	Since $X$ is complete, $\sqrt{-1}T$ is self-adjoint, and for any $u\in\Dom \sqrt{-1}T\cap\Dom\ddbar_b\cap\Dom\ddbar_b^*$ there exists a sequence $\{u_m\}_{m\geq 1}\subset \Omega_c^{\bullet,\bullet}(X)$ converges to $u$ with respect to the graph norm of $\ddbar_b+\ddbar_b^*+\sqrt{-1}T$. Thus, for $u\in\Dom(\ddbar_b)\cap\Dom(\ddbar^*_b)\cap\Dom(\sqrt{-1}T)\cap\cE(\lambda,\sqrt{-1}T)\cap L^2_{0,q}(X)$, we have
	\be
	-\lambda\|u\|^2\leq\left(-\sqrt{-1}Tu|u\right)\leq \frac{1}{q}\left(\|\ddbar_b u\|^2+\|\ddbar^*_b u\|^2\right)+C\|u\|^2.
	\ee
	
	Assume $-\lambda>C$. Let $v\in \Dom(\ddbar_b)\cap\Dom(\ddbar_b)\cap\cE(\lambda,\sqrt{-1}T)\cap L^2_{0,q}(X)$. By the completeness of $X$, there exists a sequence $\{ v_m \}_{m\geq 1}\in\Omega_c^{0,q}(X)$ such that $\|\ddbar_bv-\ddbar_bv_m\|+\|\ddbar^*_bv-\ddbar^*_bv_m\|+\|v-v_m\|\rightarrow 0$ as $n\rightarrow \infty$. It is clear that $Q_\lambda v_m\in\cE(\lambda,\sqrt{-1}T)$, and
	\be
	\|\ddbar_bQ_\lambda v_m\|=\|Q_\lambda\ddbar_b v_m\|\leq \|\ddbar_b v_m\|<\infty,\\
	\|\ddbar^*_bQ_\lambda v_m\|=\|Q_\lambda\ddbar^*_b v_m\|\leq \|\ddbar^*_b v_m\|<\infty,\\
	\|\sqrt{-1}TQ_\lambda v_m\|=\|Q_\lambda\sqrt{-1}Tv_m\|\leq \|\sqrt{-1}Tv_m\|< \infty.
	\ee
	Here the last inequality follows from that $Q_\lambda(\sqrt{-1}T)\subset (\sqrt{-1}T)Q_\lambda$, i.e., $Q_\lambda\Dom(\sqrt{-1}T)\subset \Dom(\sqrt{-1}T)$ and $Q_\lambda(\sqrt{-1}T)= (\sqrt{-1}T)Q_\lambda$ on $\Dom(\sqrt{-1}T)$. 
	Thus, we conclude that $Q_\lambda v_m\in\Dom(\ddbar_b)\cap\Dom(\ddbar^*_b)\cap\Dom(\sqrt{-1}T)\cap\cE(\lambda,\sqrt{-1}T)\cap L^2_{0,q}(X)$ for $m=1,2,\cdots$. Moreover,
\be
\begin{split}
(-\lambda-C)\|Q_\lambda v_m\|^2
&\leq \frac{1}{q}\left(\|\ddbar_b Q_\lambda v_m\|^2+\|\ddbar^*_b Q_\lambda v_m\|^2\right)\\
&=\frac{1}{q}\left(\| Q_\lambda\ddbar_b v_m\|^2+\| Q_\lambda\ddbar_b^* v_m\|^2\right)\\
&\leq \frac{1}{q}\left(\|\ddbar_b v_m\|^2+\|\ddbar^*_b v_m\|^2\right)
\end{split}
\ee
Note $\|v-v_m\|^2=\|v-Q_\lambda v_m\|^2+\|(Q_\lambda-1) v_m\|^2\rightarrow 0$ as $m\rightarrow \infty$. We obtain
	\be
	(-\lambda-C)\|v\|^2\leq \frac{1}{q}\left(\|\ddbar_b v\|^2+\|\ddbar^*_b v\|^2\right)
	\ee
	for $v\in \Dom(\ddbar_b)\cap\Dom(\ddbar_b)\cap\cE(\lambda,\sqrt{-1}T)\cap L^2_{0,q}(X)$. Finally, we use $-\lambda>C$.
\end{proof}

We note that the pervious vanishing theorems on CR
manifolds imply some classical vanishing theorems
for complete K\"ahler manifolds.
\begin{cor}[Andreotti-Vesentini]
Let $(M,\omega)$ be a complete K\"ahler manifold
and let $(L,h^L)\to M$ be a Hermitian holomorphic
line bundle such that $\sqrt{-1}R^L=\omega$
and there is $C>0$ such that 
$\sqrt{-1}R^{K^*_M}_{\omega}\geq-C\omega$ on $M$.
Then there exists $m_0\in\N$ such that for every $m\geq m_0$
we have $H^q_{(2)}(M,L^m)=0$ for $q\geq1$.
\end{cor}
\begin{proof}
We apply the previous results for the CR manifold $X$
constructed in Example \ref{ex:GT}.
In this case $T=\partial_\theta$. 
For $m\in\Z$, the space $L^2_{0,q}(M,L^m)$ is isometric to the space
of $m$-equivariant $L^2$ forms on $X$,
$L^2_{0,q}(X)_m=\{u\in L^2_{0,q}(X): (e^{i\theta})^*u=e^{im\theta}u,
\text{for any $e^{i\theta}\in S^1$}\}$.
Note that $L^2_{0,q}(X)_m=\cE^{(q)}(-m,\sqrt{-1}\partial_\theta)$
and the $L^2$-Dolbeault complex $(L^2_{0,\bullet}(M,L^m),\ddbar)$
is isomorphic to the $\ddbar_b$-complex $(L^2_{0,\bullet}(X)_m,\ddbar_b)$.
Hence the assertion follows from Theorem \ref{t-gue201113yyd}.
\end{proof}

\subsection{Szeg\H{o} kernel asymptotic expansions}\label{s-gue201115yyd}

We first introduced some notations. Let $D\subset X$ be an open coordinate patch with local coordinates $x=(x_1,\ldots,x_{2n+1})$. 
Let $m\in\mathbb R$, $0\leq\rho,\delta\leq1$. Let $S^m_{\rho,\delta}(T^*D)$
denote the H\"{o}rmander symbol space on $T^*D$ of order $m$ type $(\rho,\delta)$
and let $S^m_{{\rm cl\,}}(T^*D)$
denote the space of classical symbols on $T^*D$ of order $m$, 
see  Grigis-Sj\"{o}strand~\cite[Definition 1.1 and p.\,35]{GS94} 
and Definition~\ref{d-gue201114yyd}.
Let
$L^m_{\rho,\delta}(D)$ and
$L^m_{{\rm cl\,}}(D)$
denote the space of pseudodifferential operators on $D$ of order $m$ type $(\rho,\delta)$
and the space of classical
pseudodifferential operators on $D$ of order $m$ respectively. 

Let $\Sigma$ be the characteristic manifold of $\Box_b$. We have 
\begin{equation}\label{e-gue201115yyds}
\begin{split}
&\Sigma=\Sigma^-\bigcup\Sigma^+,\\
&\Sigma^-=\{(x,-c\omega_0(x))\in T^*X;\, c<0\},\\
&\Sigma^+=\{(x,-c\omega_0(x))\in T^*X;\, c>0\}.
\end{split}
\end{equation}
We recall the following definition induced in~\cite[Definition 2.4]{HM16}

\begin{defn}\label{d-gue201115yyds}
Let $Q:L^2(X)\rightarrow L^2(X)$ be a continuous operator. 
Let $D\Subset X$ be an open local coordinate patch of $X$ with local coordinates 
$x=(x_1,\ldots,x_{2n+1})$ and let $\eta=(\eta_1,\ldots,\eta_{2n+1})$
be the dual variables of $x$.
We write
\[\mbox{$Q\equiv0$ at $\Sigma^-\cap T^*D$}\,,\]
if for every $D'\Subset D$, 
\[
Q(x,y)\equiv\int e^{i\langle x-y,\eta\rangle}q(x,\eta)d\eta\:\:\text{on $D'$},
\]
where $q(x,\eta)\in S^0_{1,0}(T^*D')$
and there exist $M>0$ and a conic open neighbourhood $\Lambda_-$ of $\Sigma^-$
such that for every $(x,\eta)\in T^*D'\cap\Lambda_-$ with $\abs{\eta}\geq M$, we have $q(x,\eta)=0$.
\end{defn} 

For a given point $x_0\in D$, let $\{W_j\}_{j=1}^{n}$ be an
orthonormal frame of $T^{1,0}X$ with respect to  $\langle\,\cdot\,|\,\cdot\,\rangle$ near $x_0$, for which the Levi form
is diagonal at $x_0$. Put
\begin{equation}\label{levi140530}
\cL_{x_0}(W_j,\ol W_\ell)=\mu_j(x_0)\delta_{j\ell}\,,\;\; j,\ell=1,\ldots,n\,.
\end{equation}
We will denote by
\begin{equation}\label{det140530}
\det\cL_{x_0}=\prod_{j=1}^{n}\mu_j(x_0)\,.
\end{equation}

We recall the following results in~\cite[Theorems 1.9, 5.1]{HM16}.

\begin{thm}\label{t-gue201115yyds}
Let $D\Subset X$ be an open coordinate patch with local coordinates $x=(x_1,\ldots,x_{2n+1})$. 
Let $Q:L^2(X)\rightarrow L^2(X)$ be a continuous operator  
and let $Q^*$ be the $L^2$ adjoint of $Q$
with respect to $(\,\cdot\,|\,\cdot\,)$. Suppose that $\Box^{(0)}_b$ has local $L^2$ closed range
on $D$ with respect to $Q$ and $QS^{(0)}=S^{(0)}Q$ on $L^2(X)$ and 
\[\mbox{$Q-Q_0\equiv 0$ at $\Sigma^-\bigcap T^*D$},\] 
where $Q_0\in L^0_{{\rm cl\,}}(D)$.
Then,
\begin{equation}\label{e-gue201115yydt}
(Q^*S^{(0)}Q)(x,y)\equiv\int^\infty_0e^{i\varphi(x,y)t}a(x,y,t)dt\ \ \mbox{on $D$},
\end{equation}
where 
\begin{equation}\label{e-gue201115yydw}
\begin{split}
&\varphi\in \cC^\infty(D\times D),\ \ {\rm Im\,}\varphi(x, y)\geq0,\\
&\varphi(x, x)=0,\ \ \varphi(x, y)\neq0\ \ \mbox{if}\ \ x\neq y,\\
&d_x\varphi(x, y)\big|_{x=y}=\omega_0(x), \ \ d_y\varphi(x, y)\big|_{x=y}=-\omega_0(x), \\
&\varphi(x, y)=-\ol\varphi(y, x),
\end{split}
\end{equation}
$a(x, y, t)\in S^{n}_{{\rm cl\,}}\big(D\times D\times\mathbb{R}_+\big)$
and the leading term $a_0(x,y)$ of the expansion \eqref{e-gue201114yydII} 
of $a(x,y,t)$ satisfies
\begin{equation}  \label{e-gue201115yydx}
a_0(x, x)=\frac{1}{2}\pi^{-n-1}
\abs{{\rm det\,}\mathcal{L}_x}\overline{q(x,\omega_0(x))}q(x,\omega_0(x)),\ \ \mbox{for all $x\in D$},
\end{equation}
where $\det\mathcal{L}_x$ is the determinant of the Levi form defined in \eqref{det140530}, 
$q(x,\eta)\in \cC^\infty(T^*D)$ is the principal symbol of $Q$.\end{thm}

We refer the reader to~\cite[Theorems 3.3, 4.4]{HM16} for more properties for the phase $\varphi$ in \eqref{e-gue201115yydw}. 

Let $D=U\times I$ be a BRT chart with BRT coordinates $x=(x_1,\ldots,x_{2n+1})$. For $\lambda\in\mathbb R$, put 
\begin{equation}\label{e-gue201115ycdp}
\hat Q_{\tau_\lambda}:=(2\pi)^{-(2n+1)}\int e^{i<x-y,\eta>}\tau_\lambda(-\eta_{2n+1})d\eta\in L^0_{1,0}(D).
\end{equation}
It is not difficult to see that 
\begin{equation}\label{e-gue201115ycdq}
\hat Q_{\tau_\lambda}-I\equiv0\ \ \mbox{at $\Sigma^-\bigcap T^*D$}.
\end{equation}
Assume that the $\mathbb R$-action is free. From \eqref{e-gue201110yydII}, we see that $Q_{\tau_\lambda}=\hat Q_{\tau_\lambda}$ on $D$. From this 
observation, Theorem~\ref{t-gue201115yyd}, Theorem~\ref{t-gue201115yyds}, \eqref{e-gue201115ycdq} and notice that $Q^*_{\tau_\lambda}S^{(0)}Q_{\tau_\lambda}=Q_{\tau^2_\lambda}S^{(0)}$, where $Q^*_{\tau_\lambda}$ is the $L^2$ adjoint of $Q_{\tau_\lambda}$
with respect to $(\,\cdot\,|\,\cdot\,)$, we get 

\begin{thm}\label{t-gue201115ycda}
Suppose that the $\mathbb R$-action is free. 
Assume that $2\sqrt{-1}\cL$ is complete and there is $C>0$ such that 
\[\sqrt{-1}R^{K^*_X}_{\cL}\geq-2C\sqrt{-1}\cL,\ \ 
{(2\sqrt{-1}\cL)^n\wedge\omega_0\geq C\Theta_X^n\wedge\omega_0.}\]
Let $D=U\times I\Subset X$ be a BRT chart with BRT coordinates $x=(x_1,\ldots,x_{2n+1})$. 
Let $\lambda\in\mathbb R$, $\lambda\leq-2C$. Then, 
\begin{equation}\label{e-gue201115ycdk}
(Q_{\tau^2_\lambda}S^{(0)})(x,y)\equiv\int^\infty_0e^{i\varphi(x,y)t}s(x,y,t)dt\ \ \mbox{on $D$},
\end{equation}
where $\varphi\in\cC^\infty(D\times D)$ is as in \eqref{e-gue201115yydt}, $s(x, y, t)\in S^{n}_{{\rm cl\,}}\big(D\times D\times\mathbb{R}_+\big)$
and the leading term $s_0(x,y)$ of the expansion \eqref{e-gue201114yydII} 
of $s(x,y,t)$ satisfies
\begin{equation}  \label{e-gue201116yyd}
s_0(x, x)=\frac{1}{2}\pi^{-n-1}
\abs{{\rm det\,}\mathcal{L}_x},\ \ \mbox{for all $x\in D$}.
\end{equation} 
\end{thm}
 
 We now assume that the $\mathbb R$-action is not free. From \eqref{e-gue201109yyd}, we know that the $\mathbb R$-action comes from a CR torus action 
$\mathbb T^d=(e^{i\theta_1},\ldots,e^{i\theta_d})$ on $X$ and $\omega_0$, $\Theta_X$ are $\mathbb T^d$ invariant. We will use the same notations as in the discussion before Proposition~\ref{p-gue201112yyd}. We need 

\begin{lemma}\label{l-gue201116yyd}
Suppose that the $\mathbb R$-action is not free. With the notations and assumptions used above, let $D=U\times I\Subset X$ be a BRT chart with BRT coordinates $x=(x',x_{2n+1})$, 
$x'=(x_1,\ldots,x_{2n})$. Fix $D_0\Subset D$ and $\lambda\in\mathbb R$. For $u\in\cC^\infty_c(D_0)$, we have 
\begin{equation}\label{e-gue201116yydI}
\begin{split}
&Q_{\tau_\lambda}u=\hat Q_{\tau_\lambda}u+\hat R_{\tau_\lambda}u\ \ \mbox{on $D_0$},\\
&(\hat R_{\tau_\lambda}u)(x)=\frac{1}{2\pi}\sum_{(m_1,\ldots,m_d)\in\mathbb Z^d}\:
\int\limits_{\mathbb T^d}e^{i\langle x_{2n+1}-y_{2n+1},\eta_{2n+1}\rangle+i(\sum^d_{j=1}m_j\beta_j)y_{2n+1}-im_1\theta_1-\ldots-im_d\theta_d}\\
&\times\tau_\lambda(-\eta_{2n+1}) (1-\chi(y_{2n+1}))u((e^{i\theta_1},\ldots,e^{i\theta_d})\circ x')d\mathbb T_dd\eta_{2n+1}dy_{2n+1}\ \ \mbox{on $D_0$}, 
\end{split}
\end{equation}
where $\chi\in\cC^\infty_c(I)$, $\chi(x_{2n+1})=1$ for every $(x',x_{2n+1})\in D_0$ and $\beta_1\in\mathbb R,\ldots,\beta_d\in\mathbb R$ are as in \eqref{e-gue201116yyda}.
\end{lemma}

\begin{proof}
We also write $y=(y',y_{2n+1})=(y_1,\ldots,y_{2n+1})$, $y'=(y_1,\ldots,y_{2n})$, to denote the BRT coordinates $x$.
Let $u\in\cC^\infty_c(D_0)$. From \eqref{e-gue201112yydI}, It is easy to see that on $D$,
\begin{equation}\label{e-gue131217}
\begin{split}
&Q_{\tau_\lambda}u(y)\\
&=\sum_{(m_1,\ldots,m_d)\in\mathbb Z^d}\tau_\lambda(-\sum^d_{j=1}m_j\beta_j)e^{i(\sum^d_{j=1}m_j\beta_j) y_{2n+1}}\times\\
&\quad\int_{\mathbb T^d}e^{-(im_1\theta_1+\ldots+im_d\theta_d)}
u((e^{i\theta_1},\ldots,e^{i\theta_d})\circ y')dT_d. 
\end{split}
\end{equation}
Now, we claim that
\begin{equation}\label{e-gue131217III}
\hat Q_{\tau_\lambda}+\hat R_{\tau_\lambda}=Q_{\tau_\lambda}\ \ \mbox{on $\cC^\infty_0(D_0)$}.
\end{equation}
Let $u\in\cC^\infty_0(D_0)$. From  Fourier inversion formula, it is straightforward to see that
\begin{equation}\label{e-gue131217IV}
\begin{split}
&\hat Q_{\tau_\lambda}u(x)\\
&=\frac{1}{2\pi}\sum_{(m_1,\ldots,m_d)\in\mathbb Z^d}
\int e^{i\langle x_{2n+1}-y_{2n+1},\eta_{2n+1}\rangle}
\tau_\lambda(-\eta_{2n+1})\chi(y_{2n+1})\\
&\quad\times e^{i(\sum^d_{j=1}m_j\beta_j)y_{2n+1}-im_1\theta_1-\ldots-im_d\theta_d}u((e^{i\theta_1},\ldots,e^{i\theta_d})\circ x')d\mathbb T_ddy_{2n+1}d\eta_{2n+1}.
\end{split}
\end{equation}
From \eqref{e-gue131217IV} and the definition of $\hat R_{\tau_\lambda}$, we have
\begin{equation}\label{e-gue131217V}
\begin{split}
&(\hat Q_{\tau_\lambda}+\hat R_{\tau_\lambda})u(x)\\
&=\frac{1}{2\pi}\sum_{(m_1,\ldots,m_d)\in\mathbb Z^d}
\int e^{i\langle x_{2n+1}-y_{2n+1},\eta_{2n+1}\rangle}
\tau_\lambda(-\eta_{2n+1})\\
&\quad\times e^{i(\sum^d_{j=1}m_j\beta_j)y_{2n-1}-im_1\theta_1-\ldots-im_d\theta_d}u((e^{i\theta_1},\ldots,e^{i\theta_d})\circ x')d\mathbb T_ddy_{2n+1}d\eta_{2n+1}.
\end{split}\end{equation}
Note that the following formula holds for every $\alpha\in\mathbb R$,
\begin{equation}\label{dm}
\int e^{i\alpha y_{2n+1}}e^{-iy_{2n+1}\eta_{2n+1}}dy_{2n+1}=2\pi\delta_\alpha(\eta_{2n+1}),
\end{equation}
where the integral is defined as an oscillatory integral and $\delta_\alpha$ is the Dirac measure at $\alpha$.
Using \eqref{e-gue131217}, \eqref{dm} and the Fourier inversion formula,
 \eqref{e-gue131217V} becomes
\begin{equation}\label{e-gue131217VI}
\begin{split}
(\hat Q_{\tau_\lambda}+\hat R_{\tau_\lambda})u(x)
=&\sum_{(m_1,\ldots,m_d)\in\mathbb Z^d}\tau_\lambda(-\sum^d_{j=1}m_j\beta_j)e^{i(\sum^d_{j=1}m_j\beta_j)x_{2n+1}}\times\\
&\int_{\mathbb T_d}e^{-im_1\theta_1-\ldots-im_d\theta_d}u((e^{i\theta_1},\ldots,e^{i\theta_d})\circ x')d\mathbb T_d\\
=&Q_{\tau_\lambda}u(x).
\end{split}\end{equation}
From \eqref{e-gue131217VI}, the claim \eqref{e-gue131217III} follows.
\end{proof}

To study $Q_{\tau^2_\lambda}S^{(0)}$ when the $\mathbb R$ is not free, we also need the following two known results~\cite[Theorems 3.2, 5.2]{HM16}

\begin{thm} \label{t-gue201116yyda}
We assume that the $\mathbb R$-action is arbitrary. 
Let $D\Subset X$ be a coordinate patch with local coordinates 
$x=(x_1,\ldots,x_{2n+1})$. Then there exist properly supported continuous operators 
$A\in  L^{-1}_{\frac{1}{2},\frac{1}{2}}(D)$ , 
$\tilde S\in L^{0}_{\frac{1}{2},\frac{1}{2}}(D)$, such that
\begin{equation}\label{e-gue201116yydb}
\begin{split}  
&\mbox{$\Box^{(0)}_bA+\tilde S=I$ on $D$},\\
&\mbox{$A^*\Box^{(0)}_b+\tilde S=I$ on $D$},\\
&\Box^{(0)}_b\tilde S\equiv0\ \ \mbox{on $D$},\\
&A\equiv A^*\ \ \mbox{on $D$},\ \ \tilde SA\equiv0\ \ \mbox{on $D$},\\
&\tilde S\equiv \tilde S^*\equiv\tilde S^2\ \ \mbox{on $D$},
\end{split}
\end{equation}
where $A^*$, $\tilde S^*$ are the formal adjoints of $A$, $\tilde S$ 
with respect to $(\,\cdot\,|\,\cdot\,)$ respectively and $\tilde S(x,y)$ satisfies
\begin{equation}\label{e-gue201116yydw}
\tilde S(x, y)\equiv\int^{\infty}_{0}e^{i\varphi(x, y)t}s(x, y, t)dt\ \ \mbox{on $D$},
\end{equation}
where $\varphi(x,y)\in\cC^\infty(D\times D)$ and $s(x, y, t)\in S^{n}_{{\rm cl\,}}(D\times D\times\mathbb{R}_+)$ 
are as in \eqref{e-gue201115ycdk}. 
\end{thm}

\begin{thm}\label{t-gue201116yydb}
Let us consider an arbitrary $\mathbb R$-action and 
let $Q:L^2(X)\rightarrow L^2(X)$ be a continuous operator  
and let $Q^*$ be the $L^2$ adjoint of $Q$
with respect to $(\,\cdot\,|\,\cdot\,)$. Suppose that $\Box^{(0)}_b$ 
has local $L^2$ closed range
on $D$ with respect to $Q$ and $QS^{(0)}=S^{(0)}Q$ on $L^2(X)$. 
Let $D$ be a coordinate patch with local coordinates $x=(x_1,\ldots,x_{2n+1})$. We have 
\begin{equation}\label{e-gue201116yyde}
\mbox{$Q^*S^{(0)}Q\equiv\tilde S^*Q^*Q\tilde S$ on $D$}, 
\end{equation}
where $\tilde S$ is as in Theorem~\ref{t-gue201116yyda}.
\end{thm} 

We need 

\begin{lemma}\label{l-gue201116yyds}
Suppose that the $\mathbb R$-action is not free. Fix $p\in X$. Let $D=U\times I$ be a BRT chart defined near $p$ with BRT coordinates $x=(x',x_{2n+1})$, $x'=(x_1,\ldots,x_{2n})$, $x(p)=0$. Fix $D_0\Subset D$, $p\in D_0$ and $\lambda\in\mathbb R$. Then, 
\[\tilde S\hat R_{\tau_\lambda}\equiv 0\ \ \mbox{on $D_0$},\]
where $\hat R_{\tau_\lambda}$ and $\tilde S$ are as in Lemma~\ref{l-gue201116yyd} and Theorem~\ref{t-gue201116yyda} respectively. 
\end{lemma} 

\begin{proof}
From \eqref{e-gue201115yydw}, we may assume $D_0$ is small so that 
\begin{equation}\label{e-gue201117yydI}
\mbox{$|\partial_{y_{2n+1}}\varphi(x,y)|\geq C$, for every $(x,y)\in D_0$},
\end{equation}
where $C>0$ is a constant. Let $g\in\cC^\infty_c(D_0)$. From \eqref{e-gue201116yydI} and \eqref{e-gue201116yydw}, we have 
\begin{equation}\label{e-gue201117yyd}
\begin{split}
&(\tilde S\hat R_{\tau_\lambda}g)(x)\\
&=\frac{1}{2\pi}\sum_{(m_1,\ldots,m_d)\in\mathbb Z^d}\:
\int e^{it\varphi(x,u)}a(x,u,t)e^{i\langle u_{2n+1}-y_{2n+1},
\eta_{2n+1}\rangle+i(\sum^d_{j=1}m_j\beta_j)y_{2n+1}-i\sum^d_{j=1}m_j\theta_j}\\
&\times\tau_\lambda(-\eta_{2n+1}) (1-\chi(y_{2n+1}))g((e^{i\theta_1},\ldots,e^{i\theta_d})\circ u')d\mathbb T_dd\eta_{2n+1}dy_{2n+1}dv_X(u)\ \ \mbox{on $D_0$},
\end{split}
\end{equation}
where we also write $u=(u',u_{2n+1})$, $u'=(u_1,\ldots,u_{2n})$, to denote the BRT coordinates $x$. Since $u_{2n+1}\neq y_{2n+1}$, for every $(u',u_{2n+1})\in D_0$, 
$y_{2n+1}\in{\rm Supp\,}(1-\chi(y_{2n+1}))$, we can integrate by parts in $\eta_{2n+1}$ and rewrite \eqref{e-gue201117yyd}:
\begin{equation}\label{e-gue201117yydII}
\begin{split}
&(\tilde S\hat R_{\tau_\lambda}g)(x)\\
&=\frac{1}{2\pi}\sum_{(m_1,\ldots,m_d)\in\mathbb Z^d}\:
\int e^{it\varphi(x,u)}a(x,u,t)\frac{1}{i(u_{2n+1}-y_{2n+1})}\\
&\times e^{i\langle u_{2n+1}-y_{2n+1},\eta_{2n+1}\rangle+i(\sum^d_{j=1}m_j\beta_j)y_{2n+1}-i\sum^d_{j=1}m_j\theta_j}\\
&\times\tau'_\lambda(-\eta_{2n+1}) (1-\chi(y_{2n+1}))g((e^{i\theta_1},\ldots,e^{i\theta_d})\circ u')d\mathbb T_dd\eta_{2n+1}dy_{2n+1}dv_X(u)\ \ \mbox{on $D_0$}.
\end{split}
\end{equation}  
Let 
\[\begin{split}
&A(x,u',y_{2n+1})\\
&:=\int e^{it\varphi(x,u)+i\langle u_{2n+1}-y_{2n+1},\eta_{2n+1}\rangle}a(x,u,t)(1-\chi(y_{2n+1}))\frac{1}{i(u_{2n+1}-y_{2n+1})}\\
&\quad\quad\times\tau'_\lambda(-\eta_{2n+1})d\eta_{2n+1} du_{2n+1}dt.\end{split}\]
From \eqref{e-gue201117yydI}, we see that 
\[|\partial_{u_{2n+1}}(it\varphi(x,u)+i\langle u_{2n+1}-y_{2n+1},\eta_{2n+1}\rangle)|\geq ct,\ \ \mbox{for $t\gg|\lambda|$, $\eta_{2n+1}\in{\rm Supp\,}\tau'_\lambda(-\eta_{2n+1})$}\]
for every $(x,u)\in D_0\times D_0$, where $c>0$ is a constant. From this observation, we can integrate by parts in $u_{2n+1}$ and $\eta_{2n+1}$ and deduce that 
\begin{equation}\label{e-gue201117yyda}
\mbox{$A(x,u',y_{2n+1})\in\cC^\infty(D_0\times D_0\times\mathbb R)$ and $A(x,u',y_{2n+1})$ is a Schwartz function in $y_{2n+1}$}. 
\end{equation}
We have 
\begin{equation}\label{e-gue201118yyd}
\begin{split}
&(\tilde S\hat R_{\tau_\lambda}g)(x)\\
&=\frac{1}{2\pi}\sum_{(m_1,\ldots,m_d)\in\mathbb Z^d}\int A(x,u',y_{2n+1})
e^{i(\sum^d_{j=1}m_j\beta_j)y_{2n+1}-im_1\theta_1-\ldots-im_d\theta_d}\\
&\times g((e^{i\theta_1},\ldots,e^{i\theta_d})\circ u')dv(u')dy_{2n+1}d\mathbb T_d,
\end{split}
\end{equation}
where $dv_X(u)=dv(u')du_{2n+1}$. From \eqref{e-gue201117yyda} and \eqref{e-gue201118yyd}, we have 
\begin{equation}\label{e-gue201117yydb}
\begin{split}
&\|\tilde S\hat R_{\tau_\lambda}g\|_{D_0,s}\\
&\leq C\sum_{(m_1,\ldots,m_d)\in\mathbb Z^d}\int|\int_{\mathbb T^d}e^{-im_1\theta_1-\ldots-im_d\theta_d}g((e^{i\theta_1},\ldots,e^{i\theta_d}\circ u')|^2\chi(u_{2n+1})dv_X(u)\\
&= C\sum_{(m_1,\ldots,m_d)\in\mathbb Z^d}\int|\int_{\mathbb T^d}e^{-im_1\theta_1-\ldots-im_d\theta_d}g((e^{i\theta_1},\ldots,e^{i\theta_d}\circ u)|^2\chi(u_{2n+1})dv_X(u)\\
&\leq\hat C\sum_{(m_1,\ldots,m_d)\in\mathbb Z^d}\int|\int_{\mathbb T^d}e^{-im_1\theta_1-\ldots-im_d\theta_d}g((e^{i\theta_1},\ldots,e^{i\theta_d}\circ u)|^2dv_X(u)\\
&\leq\hat C_0\|g\|^2,  
\end{split}  
\end{equation}
where $\|\cdot\|_{D_0,s}$ denotes the standard Sobolev norm of order $s$ on $D_0$, $C, \hat C, \hat C_0>0$ are constants. From \eqref{e-gue201117yydb}, we deduce that 
\[\mbox{$\tilde S\hat R_{\tau_\lambda}: L^2_c(D_0)\rightarrow H^s_{{\rm loc\,}}(D_0)$ is continuous, for every $s\in\mathbb N$}.\]

Let $\triangle_X: \cC^\infty(X)\rightarrow\cC^\infty(X)$ be the standard Laplacian on $X$ induced by $\langle\,\cdot\,|\,\cdot\,\rangle$. Since 
$\langle\,\cdot\,|\,\cdot\,\rangle$ is $\mathbb T^d$ invariant, $\triangle_X$ is $\mathbb T^d$ invariant. Fix $s\in\mathbb N$. Let 
\[G_s: \cC^\infty_c(D_0)\rightarrow\cC^\infty_c(D_0)\]
be a parametrix of $\triangle_X^s$ on $D_0$ and $G_s$ is properly supported on $D_0$. Let $g\in\cC^\infty_c(D_0)$. We have 
\begin{equation}\label{e-gue201118yyda}
g=(\triangle^s_XG_s+F_s)g,\ \ F_s\equiv 0\ \ \mbox{on $D_0$}. 
\end{equation}
Now, on $D_0$, 
\begin{equation}\label{e-gue201118yydb}
\tilde S\hat R_{\tau_\lambda}g=\tilde S\hat R_{\tau_\lambda}(\triangle^s_XG_sg)+\tilde S\hat R_{\tau_\lambda}(F_sg).
\end{equation}
Since $F_s$ is smoothing, we have 
\begin{equation}\label{e-gue201118yydc}
\|\tilde S\hat R_{\tau_\lambda}(F_sg)\|_{D_0,s}\leq C\|g\|_{-s},
\end{equation}
where $C>0$ is a constant. Now, we can integrate by parts and repeat the proof of \eqref{e-gue201118yyd} and show that 
\begin{equation}\label{e-gue201118yydd}
\begin{split}
&(\tilde S\hat R_{\tau_\lambda}\triangle^s_XG_sg)(x)\\
&=\frac{1}{2\pi}\sum_{(m_1,\ldots,m_d)\in\mathbb Z^d}\int A_s(x,u',y_{2n+1})
e^{i(\sum^d_{j=1}m_j\beta_j)y_{2n+1}-im_1\theta_1-\ldots-im_d\theta_d}\\
&\times (G_sg)((e^{i\theta_1},\ldots,e^{i\theta_d})\circ u')dv(u')dy_{2n+1}d\mathbb T_d,
\end{split}
\end{equation}
where 
\begin{equation}\label{e-gue201118yyde}
\mbox{$A_s(x,u',y_{2n+1})\in\cC^\infty(D_0\times D_0\times\mathbb R)$ and $A(x,u',y_{2n+1})$ is a Schwartz function in $y_{2n+1}$}. 
\end{equation}
From \eqref{e-gue201118yydd} and \eqref{e-gue201118yyde}, we can repeat the proof of \eqref{e-gue201117yydb} and conclude that
\begin{equation}\label{e-gue201118yydf}
\|\tilde S\hat R_{\tau_\lambda}(\triangle^s_XG_sg)\|_{D_0,s}\leq C\| G_sg\|\leq C_1\|g\|_{-s},
\end{equation}
where $C, C_1>0$ are constants. From \eqref{e-gue201118yyda}, \eqref{e-gue201118yydb}, \eqref{e-gue201118yydc} and \eqref{e-gue201118yydf}, we get that 
\[\mbox{$\tilde S\hat R_{\tau_\lambda}: H^{-s}_{{\rm comp\,}}(D_0)\rightarrow H^s_{{\rm loc\,}}(D_0)$ is continuous, for every $s\in\mathbb N$}.\]
Hence, $\tilde S\hat R_{\tau_\lambda}$ is smoothing on $D_0$. 
\end{proof}

Now, we can prove: 

\begin{thm}\label{t-gue201118yyd}
Suppose that the $\mathbb R$-action is not free. 
Assume that $2\sqrt{-1}\cL$ is complete and there is $C>0$ such that 
\[\sqrt{-1}R^{K^*_X}_{\cL}\geq-2C\sqrt{-1}\cL,\ \ 
{(2\sqrt{-1}\cL)^n\wedge\omega_0\geq C\Theta_X^n\wedge\omega_0.}\]
Let $D=U\times I\Subset X$ be a BRT chart with BRT coordinates 
$x=(x_1,\ldots,x_{2n+1})$. 
Let $\lambda\in\mathbb R$, $\lambda\leq-2C$. Then, 
\begin{equation}\label{e-gue201118yydk}
(Q_{\tau^2_\lambda}S^{(0)})(x,y)\equiv\int^\infty_0e^{i\varphi(x,y)t}s(x,y,t)dt\ \ 
\mbox{on $D$},
\end{equation}
where $\varphi\in\cC^\infty(D\times D)$ is as in 
\eqref{e-gue201115yydt}, 
$s(x, y, t)\in S^{n}_{{\rm cl\,}}\big(D\times D\times\mathbb{R}_+\big)$ 
is as in \eqref{e-gue201115ycdk}. 
\end{thm} 

\begin{proof}
From \eqref{e-gue201116yydI}, \eqref{e-gue201116yyde} 
and Lemma~\ref{l-gue201116yyds}, we see that on $D$, 
\[Q_{\tau^2_\lambda}S^{(0)}\equiv
\tilde S^*\hat Q_{\tau\lambda}^*\hat Q_{\tau\lambda}\tilde S.\]
From this observation, we can repeat the proof 
of~\cite[Theorem 5.8]{HM16} and get the theorem. 
\end{proof}

\begin{thm}\label{t-gue201118yyda}
Let us consider an arbitrary $\mathbb R$-action. 
Assume that $2\sqrt{-1}\cL$ is complete and there is $C>0$ such that 
\[\sqrt{-1}R^{K^*_X}_{\cL}\geq-2C\sqrt{-1}\cL,\ \ 
{(2\sqrt{-1}\cL)^n\wedge\omega_0\geq C\Theta_X^n\wedge\omega_0.}\]
Let $D=U\times I$ be a BRT chart with BRT coordinates $x=(x_1,\ldots,x_{2n+1})$. 

Let $\lambda\in\mathbb R$, $\lambda\leq-2C$. Then, 
\begin{equation}\label{e-gue201118ycda}
(I-Q_{\tau^2_\lambda})^2S^{(0)}\equiv0\ \ \mbox{on $D$}.
\end{equation}
\end{thm} 

\begin{proof}
From \eqref{e-gue201116yydb}, we have 
\begin{equation}\label{e-gue201118ycdh}
(I-Q_{\tau^2_\lambda})S^{(0)}=\tilde S(I-Q_{\tau^2_\lambda})S^{(0)}.
\end{equation}
From Lemma~\ref{l-gue201116yyds}, we have 
\[\tilde S(I-Q_{\tau^2_\lambda})\equiv\tilde S(I-\hat Q_{\tau^2_\lambda})\ \ \mbox{on $D_0$}. \]
Since ${\rm WF\,}(I-\hat Q_{\tau^2_\lambda})\bigcap\Sigma^-=\emptyset$ and ${\rm WF'\,}(S)={\rm diag\,}(\Sigma^-\times\Sigma^-)$, we have 
\begin{equation}\label{e-gue201118ycdj}
\tilde S(I-\hat Q_{\tau^2_\lambda})\equiv 0\ \ \mbox{on $D_0$},
\end{equation}
where ${\rm WF\,}(I-\hat Q_{\tau^2_\lambda})$ denotes the wave front set of $I-\hat Q_{\tau^2_\lambda}$ and 
\[ {\rm WF'\,}(\tilde S)=\{(x,\xi,y,\eta)\in T^*D\times T^*D;\, (x,\xi,y,-\eta)\in{\rm WF'\,}(\tilde S).\]
From \eqref{e-gue201118ycdh} and \eqref{e-gue201118ycdj}, w get 
\begin{equation}\label{e-gue201118yydu}
\mbox{$(I-Q_{\tau^2_\lambda})S^{(0)}: L^2(X)\rightarrow\cC^\infty(D)$ is continuous}
\end{equation}
and hence 
\begin{equation}\label{e-gue201118yydv}
\mbox{$S^{(0)}(I-Q_{\tau^2_\lambda}): \mathscr E'(D)\rightarrow L^2(X)$ is continuous}.
\end{equation}
From \eqref{e-gue201118yydu} and \eqref{e-gue201118yydv}, we get 
\[\mbox{$(I-Q_{\tau^2_\lambda})S^{(0)}(I-Q_{\tau^2_\lambda}):\mathscr E'(D)\rightarrow \cC^\infty(D)$ is continuous}.\]
The theorem follows. 
\end{proof}
        
We can now prove the main result of this work. 

\begin{thm}[=Theorem \ref{thm-main}]\label{t-gue201118yydp}
The $\mathbb R$-action can be free or not free. 
Assume that $2\sqrt{-1}\cL$ is complete and there is $C>0$ such that 
\[\sqrt{-1}R^{K^*_X}_{\cL}\geq-2C\sqrt{-1}\cL,\ \ 
{(2\sqrt{-1}\cL)^n\wedge\omega_0\geq  C\Theta_X^n\wedge\omega_0.}\]  
Let $D\Subset X$ be a local coordinate patch with local coordinates $x=(x_1,\ldots,x_{2n+1})$. 
Then, 
\begin{equation}\label{e-gue201118yydq}
S^{(0)}(x,y)\equiv\int^\infty_0e^{i\varphi(x,y)t}s(x,y,t)dt\ \ \mbox{on $D$},
\end{equation}
where $\varphi\in\cC^\infty(D\times D)$, 
\begin{equation}\label{e-gue201118yydz}
\begin{split}
&\varphi\in \cC^\infty(D\times D),\ \ {\rm Im\,}\varphi(x, y)\geq0,\\
&\varphi(x, x)=0,\ \ \varphi(x, y)\neq0\ \ \mbox{if}\ \ x\neq y,\\
& d_x\varphi(x, y)\big|_{x=y}=\omega_0(x), \ \ d_y\varphi(x, y)\big|_{x=y}=-\omega_0(x), \\
&\varphi(x, y)=-\ol\varphi(y, x),
\end{split}
\end{equation}
$s(x, y, t)\in S^{n}_{{\rm cl\,}}\big(D\times D\times\mathbb{R}_+\big)$
and the leading term $s_0(x,y)$ of the expansion \eqref{e-gue201114yydII} 
of $s(x,y,t)$ satisfies
\begin{equation}  \label{e-gue201118yydr}
s_0(x, x)=\frac{1}{2}\pi^{-n-1}
\abs{{\rm det\,}\mathcal{L}_x},\ \ \mbox{for all $x\in D$}.
\end{equation} 
 \end{thm} 
 
 \begin{proof}
 Let $\lambda\in\mathbb R$, $\lambda\leq-2C$. Form Theorem~\ref{t-gue201118yyda}, we have 
 \begin{equation}\label{e-gue201118ycdab}
 (I-2Q_{\tau^2_\lambda}+Q_{\tau^4_\lambda})S^{(0)}\equiv0\ \ \mbox{on $D$}. 
 \end{equation}
 We can repeat the proofs of Theorem~\ref{t-gue201115ycda} and Theorem~\ref{t-gue201118yyd} and get 
 \begin{equation}\label{e-gue201118yydab}
 \begin{split}
 &Q_{\tau^2_\lambda}S^{(0)}\equiv\int^\infty_0e^{i\varphi(x,y)t}\hat s(x,y,t)\ \ \mbox{on $D$},\\
  &Q_{\tau^4_\lambda}S^{(0)}\equiv\int^\infty_0e^{i\varphi(x,y)t}\tilde s(x,y,t)\ \ \mbox{on $D$},
 \end{split}
 \end{equation}
 where $\varphi\in\cC^\infty(D\times D)$ is as in \eqref{e-gue201115yydt}, $\hat s(x, y, t), \tilde s(x,y,t)\in S^{n}_{{\rm cl\,}}\big(D\times D\times\mathbb{R}_+\big)$ are as in \eqref{e-gue201115ycdk}. From \eqref{e-gue201118ycdab} and \eqref{e-gue201118yydab}, the theorem follows. 
 \end{proof}

\section{Examples}
We now consider  Heisenberg group $\field{H}=\C^n\times \R$ with  CR structure 
	\be
	T^{1,0}\field{H}:={\rm span\,}\left\{\frac{\dbar}{\dbar z_j}+i\frac{\dbar\phi}{\dbar z_j}(z)\frac{\dbar}{\dbar x_{2n+1}}\right\}_{j=1}^n,
	\ee
	where $\phi\in\cC^\infty(\C^n,\R)$. Let $(\,\cdot\,|\,\cdot\,)_{\field{H}}$ be the $L^2$ inner product on $\field{H}$ induced by the Euclidean measure $dx$ on $\mathbb R^{2n+1}$. Let 
	\[S_{\field{H}}: L^2(\field{H})\rightarrow\{u\in L^2(\field{H});\, (\frac{\dbar}{\dbar\overline z_j}-i\frac{\dbar\phi}{\dbar\overline z_j}(z)\frac{\dbar}{\dbar x_{2n+1}})u=0\}\]
	be the orthogonal projection with respect to $(\,\cdot\,|\,\cdot\,)_{\field{H}}$ and let $S_{\field{H}}(x,y)\in\mathscr D'(\field{H}\times\field{H})$ be the distribution kernel of $S_{\field{H}}$. From Theorem~\ref{thm-main}, we deduce 

\begin{cor}\label{c-gue201205yyd}
With the notations used above, assume that $\left(\frac{\partial^2\phi(z)}{\partial z_j\partial\overline z_k}\right)^n_{j,k=1}$ is positive definite, at every $z\in\mathbb C^n$. Let $0<\lambda_1(z)\leq\ldots\leq\lambda_n(z)$ be the eigenvalues of $\left(\frac{\partial^2\phi(z)}{\partial z_j\partial\overline z_k}\right)^n_{j,k=1}$, for every $z\in\mathbb C^n$. 
Suppose that there is $C>0$ such that 
	\begin{equation}	\label{e-gue201205yyd}      
	\begin{split}
	&\sqrt{-1}\dbar\ddbar\left(-\log\det\left(\frac{\partial^2\phi}{\partial z_j\partial\overline z_k}\right)^n_{j,k=1}\right)\geq -C\sqrt{-1}\dbar\ddbar\phi\ \ ,\\
	&\frac{1}{\lambda_1(z)}\leq C,\ \ \mbox{for every $z\in\mathbb C^n$}.
	\end{split} 
	\end{equation}
Let $D\Subset\field{H}$ be any open set.  	Then, 
	\begin{equation}\label{e-gue201205yydI}
		S_{\field{H}}(x,y)\equiv\int^\infty_0e^{i\varphi(x,y)t}s(x,y,t)dt\ \ \mbox{on $D$},
	\end{equation}
	where $\varphi\in\cC^\infty(D\times D)$ and $s(x, y, t)\in S^{n}_{{\rm cl\,}}\big(D\times D\times\mathbb{R}_+\big)$ are as in Theorem~\ref{thm-main}. 
	\end{cor}

\begin{exam}   
      
     With the notations used in Corollary~\ref{c-gue201205yyd}, assume that 
     	\be
     	\phi(z)=|z|^2+r(z),
     	\ee
     	with $r(z)\in \cC^\infty_c(\C^n)$ and $\sqrt{-1}\dbar\ddbar(|z|^2+r(z))>0$ on $\C^n$. With this $\phi$, we can check the conditions of Corollary \ref{c-gue201205yyd} fulfilled as follows.
        In fact, in this case, we have 
     	$$\det\left(\frac{\partial^2\phi}{\partial z_j\partial\overline z_k}\right)^n_{j,k=1}=\det\left(\frac{\partial^2(|z|^2+r(z))}{\partial z_j\partial\overline z_k}\right)^n_{j,k=1}=1+F(z)>0$$ with some $F(z)\in \cC^\infty_c(\C^n)$. And we have
     	$$
     	\sqrt{-1}\dbar\ddbar\left(-\log\det\left(\frac{\partial^2\phi}{\partial z_j\partial\overline z_k}\right)^n_{j,k=1}\right)=\sqrt{-1}\dbar\ddbar\left(-\log (1+F(z))\right)\in\Omega_c^{1,1}(\C^n).
     	$$
     Since $r(z)\in \cC^\infty_c(\C^n)$ and $\sqrt{-1}\dbar\ddbar\phi=\sqrt{-1}\dbar\ddbar(|z|^2+r(z))>0$, we have a uniform lower bound for the smallest eigenvalue, i.e.,  $\lambda_1(z)>1/C_1$ for some $C_1>0$. Moreover, we can choose $C_2>0$ sufficiently large such that
     $$
     \sqrt{-1}\dbar\ddbar\left(-\log (1+F(z))\right) +C_2\sqrt{-1}\dbar\ddbar\phi\geq 0,
     $$
     since the first term $\sqrt{-1}\dbar\ddbar\left(-\log (1+F(z))\right)$ is a real $(1,1)$-form with compact support in $\C^n$ and the second term $\sqrt{-1}\dbar\ddbar\phi$ is a real positive $(1,1)$-form with a uniformly positive lower bound for the smallest eigenvalue $\lambda_1(z)>1/C_1$ on $\C^n$. Finally we obtain $C:=\max\{C_1,C_2\}>0$ as desired in $(\ref{e-gue201205yyd})$.

      With this $\phi$, it is easy to see that \eqref{e-gue201205yyd} hold. 
    This example shows that, after small perturbation of the Levi form of Heisenberg group, we still can obtain the Szeg\H{o} kernel expansion via Corollary~\ref{c-gue201205yyd}.  \end{exam}

\begin{exam}
Let $(X,T^{1,0}X)$ be a strictly pseudoconvex, 
CR manifold of dimension $2n+1$, $n\geq 1$, with a discrete, proper, CR action 
$\Gamma$ such that the quotient $X/\Gamma$ is compact. 
Assume $X$ admits a transversal CR $\R$-action on $X$ and let $\Theta_X$ 
be a $\Gamma$-invariant, $\R$-invariant, Hermitian metric on $X$. 
Then the conclusion of Theorem \ref{thm-main} holds. 
In fact, $\Gamma$-covering manifold is complete and we can find 
the desired constant $C$ depending on the fundamental domain 
$U\Subset X$ given by the $\Gamma$-action such that (\ref{e-cond}) fulfilled. 
As a consequence, if we consider the circle bundle case in which $R^L=2\cL$, 
we could obtain the Bergman kernel expansion for covering manifold \cite[6.1.2]{MM}.   
\end{exam} 

\textbf{Acknowledgements.}
Chin-Yu Hsiao was partially supported by Taiwan Ministry of Science and
Technology projects 108-2115-M-001-012-MY5,  109-2923-M-001-010-MY4 and Academia Sinica Career Development Award. George Marinescu partially supported by 
the DFG funded project SFB TRR 191 `Symplectic Structures in Geometry, 
Algebra and Dynamics' (Project-ID 281071066\,--\,TRR 191).

\bibliographystyle{alpha}    
  
\end{document}